\documentclass[11pt]{amsart}
\usepackage{amsmath}
\usepackage{amsfonts}
\usepackage{amssymb}
\usepackage{amsthm}
\usepackage{enumerate}
\usepackage[totalwidth=450pt, totalheight=650pt, centering]{geometry}

\usepackage{bbm}
\usepackage[mathscr]{eucal}
\usepackage{tikz}

\usepackage{graphicx}
\usepackage{verbatim}

\usepackage[T2A,T1]{fontenc}

\newcommand{\les}{\lesssim}

\newcommand{\ext}{{\text{ext}}}
\newcommand{\Wmq}{\fW^\mu(Q)}
\newcommand{\Rkmw}{\cR^\mu_k(W)}
\newcommand{\Wgood}{\mathfrak W^\mu_{\text{\rm good}}}
\newcommand{\Wbad}{\mathfrak W^\mu_{\text{\rm bad}}}

\newcommand{\w}{}
\newcommand{\q}{}

\newcommand\bbone{{\mathbbm 1}}

\newcommand{\R}{{\mathbb{R}}}

\newcommand{\ta}{{\widetilde {a}}}

\newcommand{\sF}{\mathscr{F}}

\newcommand{\supp}[1]{\mathrm{supp}\q(#1\w)}

\newcommand{\sJ}{\mathscr J}

\newcommand{\K}{\mathcal K}

\newcommand{\Be}{\begin{equation}}
\newcommand{\Ee}{\end{equation}}

\newcommand{\Bm}{\begin{multline}}
\newcommand{\Em}{\end{multline}}
\newcommand{\Bea}{\begin{eqnarray}}
\newcommand{\Eea}{\end{eqnarray}}
\newcommand{\Beas}{\begin{eqnarray*}}
\newcommand{\Eeas}{\end{eqnarray*}}
\newcommand{\Benu}{\begin{enumerate}}
\newcommand{\Eenu}{\end{enumerate}}
\newcommand{\Bi}{\begin{itemize}}
\newcommand{\Ei}{\end{itemize}}

\def\sD{{\mathscr {D}}}

\def\intslash{\rlap{\kern  .32em $\mspace {.5mu}\backslash$ }\int}
\def\qsl{{\rlap{\kern  .32em $\mspace {.5mu}\backslash$ }\int_{Q_x}}}
\def\Re{\operatorname{Re\,}}

\def\vth{\vartheta}

\def\emph#1{{\it #1 }}

\def\diam{{\,\mathrm{diam}}}

\def\ga{\gamma}

\def\cf{{\it cf}}

\def\ext{{\mathrm{ext}}}
\def\dist{{\mathrm{dist}}}

\def\rad{{\mathrm{ rad}}}

\def\inn#1#2{\langle#1,#2\rangle}
\def\biginn#1#2{\big\langle#1,#2\big\rangle}

\def\noi{\noindent}

\def\meas{{\mathrm{ meas}}}

\def\lc{\lesssim}
\def\gc{\gtrsim}

\def\eps{\varepsilon}

\def\la{\lambda}

\def\vphi{\varphi}
             
\def\om{\omega}              \def\Om{\Omega}

\def\fM{{\mathfrak {M}}}

\def\fQ{{\mathfrak {Q}}}

\def\fS{{\mathfrak {S}}}

\def\fW{{\mathfrak {W}}}

\def\bbC{{\mathbb {C}}}

\def\bbN{{\mathbb {N}}}

\def\bbR{{\mathbb {R}}}

\def\bbT{{\mathbb {T}}}

\def\bbZ{{\mathbb {Z}}}
\def\cA{{\mathcal {A}}}

\def\cF{{\mathcal {F}}}

\def\cL{{\mathcal {L}}}
\def\cM{{\mathcal {M}}}
\def\cN{{\mathcal {N}}}
\def\cO{{\mathcal {O}}}

\def\cQ{{\mathcal {Q}}}
\def\cR{{\mathcal {R}}}
\def\cS{{\mathcal {S}}}

\def\cZ{{\mathcal {Z}}}

\def\be#1{\begin{equation}\label{ #1}}
\def\endeq{\end{equation}}
\def\endal{\end{align}}
\def\bas{\begin{align*}}
\def\eas{\end{align*}}
\def\bi{\begin{itemize}}
\def\ei{\end{itemize}}
\def\eps{\varepsilon}

\def\emph#1{{\it #1}}
\def\textbf#1{{\bf #1}}

\newtheorem{thm}{Theorem}[section]
\newtheorem{cor}[thm]{Corollary}
\newtheorem{prop}[thm]{Proposition}
\newtheorem{lemma}[thm]{Lemma}

\newtheorem*{namedtheorem}{\theoremname}
\newcommand{\theoremname}{testing}

\theoremstyle{remark}

\theoremstyle{definition}

\theoremstyle{remark}

\usepackage[hidelinks]{hyperref} 
\numberwithin{equation}{section}
\begin{document}

\title[Riesz means: Strong summability at the critical index]{Riesz means of Fourier series and integrals: \\
Strong summability at the critical index}
\author{Jongchon Kim  \ \ \ \ \ \ \ \ \ \ \ \ \  Andreas Seeger}

\address{Jongchon Kim, Institute for Advanced Study \\ 1 Einstein  Drive\\ Princeton, NJ 08540, USA} 

\curraddr{Department of Mathematics, University of British Columbia, 1984 Mathematics Road, Vancouver, BC, Canada V6T 1Z2}

\email{jkim@math.ubc.ca}
\address{Andreas Seeger, Department of Mathematics \\ University of Wisconsin \\480 Lincoln Drive\\ Madison, WI, 53706, USA} \email{seeger@math.wisc.edu}

\begin{abstract} 
We consider spherical Riesz means of multiple Fourier series and some generalizations. 
While almost everywhere convergence of Riesz means at the critical index $(d-1)/2$ may fail for functions in the Hardy space $h^1(\mathbb T^d)$, we prove 
sharp positive results for strong summability almost everywhere. 
For functions in $L^p(\mathbb T^d)$, $1<p<2$,  we consider Riesz means at the critical index $d(1/p-1/2)-1/2$ and prove an almost 
sharp theorem on strong summability. The results follow via transference from corresponding results for Fourier integrals.  We include an endpoint bound  on maximal operators associated with  generalized Riesz means on Hardy spaces $H^p(\R^d)$ for $0<p<1$.
 \end{abstract}


\thanks{
Supported in part by the National Science Foundation  grants DMS-1500162, DMS-1638352, and DMS-1764295. 
Part of this work was 
supported  by NSF grant DMS-1440140 while the authors were in residence at the Mathematical Sciences Research Institute in Berkeley, California, during the Spring 2017 semester. The authors thank the referees for their suggestions.}

\maketitle



\setcounter{tocdepth}{2}

\section{Introduction}
We consider multiple Fourier series of functions on $\bbT^d=\bbR^d/\bbZ^d$.  
For $\ell\in \bbZ^d$ let $e_\ell(x)=e^{2\pi i \inn{x}{\ell}}$ and define the Fourier coefficients of 
$f\in L^1(\bbT^d)$ by   $\inn{f}{e_\ell}= \int_{\bbT^d} f(y) e^{-2\pi i \inn{y}{\ell}}dy$. 
We shall examine the pointwise behavior of  (generalized) Riesz means of the Fourier  series.
Fix a homogeneous distance function $\rho$, continuous on  $\bbR^d$, positive and $C^\infty$ 
on $\bbR^d\setminus \{0\}$, and satisfying, for some $b>0$,
$\rho (t^b \xi)=t\rho(\xi)$ for all $\xi\in \bbR^d$.
For $f\in L^1(\bbT^d)$ define the Riesz means of index $\la$ with respect to $\rho$, by
\Be\label{rieszmeans}\cR_t^\la f= \sum_{\substack{\ell\in \bbZ^d:\\ \rho(\ell/t)\le 1}}\big (1-{\rho(\ell /t)}\big)^\la\inn{f}{e_\ell} e_\ell.\Ee
The classical Riesz  means are recovered for $\rho(\xi)=|\xi|$,  and when in addition $\la=1$ we obtain the  Fej\'er means.
The Bochner-Riesz means are covered with $b=1/2$ by taking  $\rho(\xi)=|\xi|^2$.

It is well known via classical results for  Fourier integrals (\cite{stein-acta58}, \cite{stein-weiss-book},  \cite{seeger-archiv})  and  transference (\cite{deleeuw}, \cite{kenig-tomas},  \cite{abb})
that 
for $\la>\frac{d-1}2$ and $f\in L^1(\bbT^d)$ we have  
$\lim_{t\to\infty}\cR^\la_t f= f$,  both with respect to the $L^1$ norm  and also almost everywhere.
For  the critical index   $\la=\frac{d-1}{2}$,  
it is known that the Riesz means are of weak type $(1,1)$ and one has convergence in measure (\cite{christ-wt}, \cite{christ-sogge}) but 
Stein \cite{stein-annals61} 
showed early  that a.e. convergence may fail (see also \cite{stein-weiss-book}). Indeed, extending ideas of Bochner,  he proved the existence of an $L^1(\bbT^d)$ function for which the Bochner-Riesz means at index $\frac{d-1}2$
diverge almost everywhere, as $t\to\infty$. Stein's theorem can be seen as an analogue of the  theorem by Kolmogorov \cite{kolmogorov} on the failure of a.e. convergence for Fourier series in $L^1(\bbT)$, see \cite[ch. VIII-4]{zygmund}. 
Later,  Stein \cite{stein-cortona} proved a stronger result showing that even for some functions in the subspace $h^1(\bbT^d)$ (the local  Hardy space)  the Bochner-Riesz means at the critical index may diverge 
almost everywhere. It is then  natural to ask  what happens if we replace  almost everywhere convergence with the weaker  notion of strong convergence a.e. (also known as  strong summability a.e.) which goes back to Hardy and Littlewood \cite{hardy-littlewood}.



\noi{\bf Definition.}
  Let $0<q <\infty$. Given a measurable function $g:(0,\infty)\to \bbC$
  we say that $g(t)$ {\it converges  $q$-strongly} to $a$, as $t\to\infty$,   if 
\[\lim_{T\to \infty} \Big(\frac 1T\int_0^T |g(t)-a|^q \, dt\Big)^{1/q}=0.\]

If $g(t)$ refers to the partial sum of a series then one also says that the series is strongly  $H_q$ summable.
Clearly if $\lim_{t\to \infty} g(t)=a$ then $g(t)$ converges $q$-strongly to $a$ for all $q<\infty$. 
Vice versa if $g(t)$ converges $q$-strongly to $a$ for some $q>0$ then 
$g(t)$ is {\it almost convergent} to $a$ as $t\to \infty$.
That is, there is a (density one) subset $E \subset [0,\infty)$ satisfying  
\Be \label{eqn:almost}
\lim_{T\to \infty} \frac{ |E\cap [0,T]|}{T}=1  \;\; \text{ and } \;\; \lim_{\substack{t\to\infty\\t\in E}} g(t)=a. 
\Ee
See  \cite[ch.XIII, (7.2)]{zygmund} and also Corollary \ref{almostalmost} below.

For the classical case of a Fourier series of an $L^1(\bbT)$ function, 
Zygmund \cite{zygmund-plms} proved that the partial sum 
$ \sum_{|l|\leq t} \inn{f}{e_\ell} e_\ell(x) $ converges $q$-strongly to $f(x)$ as $t \to \infty$ a.e. for all $q<\infty$, extending an earlier result by Marcinkiewicz \cite{marcinkiewicz} for $q=2$.
Zygmund
used complex methods, but in more recent papers one can find alternative approaches with stronger results and some weaker extensions to rectangular partial sums of multiple Fourier series; see, e.g., \cite{rodin} and \cite{bwilson} and references therein. See also \cite{konyagin} for an overview of recent developments on topics related to the convergence of Fourier series.

Regarding spherical partial sums of multiple Fourier series, $q$-strong convergence results have been available for $L^p(\bbT^d)$ functions for the Bochner-Riesz means of index $\la>\la(p)$ when $p\leq 2$, $q=2$, where $\la(p) = d(\frac 1p -\frac 12)-\frac 12$ is the critical index (\cf. \cite{stein-acta58}, \cite{sunouchi}). The question of strong convergence  a.e.  for the Bochner-Riesz means at the critical index $\la(1)=\frac{d-1}{2}$, for either  $f\in L^1(\bbT^d)$ or $f\in h^1(\bbT^d)$ had been left open and was posed by S. Lu  in the survey article \cite{lu-conj}. 
We answer this question in the affirmative for $f\in h^1(\bbT^d)$ for generalized Riesz means with \emph{any} distance function $\rho$ under consideration.


\medskip


\begin{thm} \label{H1thmTd} Let  $q<\infty$ and  $\la(1)=\frac{d-1}{2}$. Then, for all $f\in h^1(\bbT^d)$ the following statements hold.
\begin{enumerate}[(i)]
\item There is a constant $C$ such that for all $\alpha>0$, 
\[ \meas\Big (\Big\{x: \sup_{T>0} \Big(\frac 1T \int_0^T|\cR^{\la(1)}_t \!f(x) |^q dt\Big)^{1/q} >\alpha\Big\}\Big) 
\le C\alpha^{-1}\|f\|_{h^1}. \]
\item 
\[ \lim_{T\to \infty} \big(\frac 1T\int_0^T|\cR^{\la(1)}_t \!f(x)-f(x)  |^q dt\Big)^{1/q} = 0 \; \text{ for almost every $x\in \bbT^d$.} \]
\end{enumerate}
 
\end{thm}

We remark that for the classical Riesz means (or generalized Riesz means assuming finite type conditions on the cosphere $\Sigma_\rho = \{ \xi: \rho(\xi)=1 \}$), Theorem \ref{H1thmTd} 
for the range $q\le 2$ could have been extracted from \cite{seeger-indiana}, although that result is not explicitly stated there. The full range $q<\infty$ obtained here seems to be new. Regarding the question posed for $f\in L^1(\bbT^d)$, in Section \ref{L1resultsect}, we derive some weaker results including $q$-strong convergence up to passing to a subsequence. 

We now address the question of strong convergence of Riesz means for $L^p(\bbT^d)$ functions at the critical index $\la=\la(p)$. In this case, $q$-strong convergence results may fail for large $q$. Our next result identifies nearly sharp range of $q$ for which $\cR^{\la(p)}_t \!f(x)$ converges $q$-strongly to $f(x)$ almost everywhere for any $f\in L^p(\bbT^d)$. We denote by $p'=\frac{p}{p-1}$ the exponent dual to $p$.


\begin{thm}\label{LpthmTd} Let $1<p<2$,
 $q<p'$ and $\la(p)=d(\frac 1p-\frac 12)-\frac 12$.  Then, for all $f\in L^p(\bbT^d)$ the following statements hold.
\begin{enumerate}[(i)]
\item  There is a constant $C$ such that for all $\alpha>0$,  \[
\meas\Big(\Big\{x\in \bbT^d: \sup_{T>0} \Big(\frac 1T \int_0^T|\cR^{\la(p)} _t\! f(x) |^q dt\Big)^{1/q} >\alpha\Big\}\Big) 
\le C\alpha^{-p}\|f\|_{L^p(\bbT^d)}.
\] 
\item 
\[ \lim_{T\to \infty} \big(\frac 1T\int_0^T|\cR^{\la(p)}_t \!f(x)-f(x)  |^q dt\Big)^{1/q} = 0 \; \text{ for almost every $x\in \bbT^d$.} \]
\item For suitable $f\in L^p(\bbT^d)$ statements (i), (ii) fail when $q>p'$.
\end{enumerate}
\end{thm}

Part (ii) in both  theorems follow by a standard argument from the respective part (i), using the fact that 
pointwise (in fact uniform)   convergence holds for Schwartz functions. We note 
that Theorem \ref{H1thmTd}  is sharp in view of  the above mentioned  example by Stein.
Moreover, part (iii) of Theorem \ref{LpthmTd} shows that the result is essentially sharp for all $p\in (1,2)$, but the  case  $q=p'$ remains open. 

We state a special case of Theorem \ref{LpthmTd} for $\la(p) = 0$, i.e., for the case of generalized spherical partial sums of Fourier series as a corollary.
\begin{cor} Let $d\ge 2$, 
 $q<\frac{2d}{d-1}$ and  $f\in L^{\frac{2d}{d+1}}(\bbT^d)$. Then 
$$\lim_{T\to \infty } \big(\frac 1T\int_0^T\Big| \sum_{\rho(\ell/t)\le 1} \inn{f}{e_\ell}e_\ell(x) - f(x) \Big  |^q {dt}\Big)^{1/q} =0 \;\; \text{ for almost every } x\in \bbT^d.$$ 
In particular, for almost every $x\in \bbT^d$, the partial sums $\sum_{\rho(\ell/t)\le 1} \inn{f}{e_\ell}e_\ell(x)$ are  almost convergent to $f(x)$ as $t\to \infty$, in the sense of \eqref{eqn:almost}.
\end{cor}

We remark that there are analogues of above results for generalized Riesz means of Fourier integral in $\R^d$ :
\Be\label{rieszmeansRd}R_t^\la f(x) = \int_{\substack{ \rho(\xi/t)\le 1}}\big (1-{\rho(\xi /t)}\big)^\la\widehat f(\xi)\,e^{2\pi i\inn \xi x}d\xi .\Ee
See \S\ref{mainwtsect}. Indeed, we derive Theorems \ref{H1thmTd} and \ref{LpthmTd} from corresponding theorems for Fourier integrals in $\bbR^d$ using transference arguments. 
Our proof uses somewhat technical arguments on atomic decomposition and 
Calder\'{o}n-Zygmund theory. Unlike the proofs of the $L^p$ boundedness of Bochner-Riesz means (such as in \cite{stein-beijing}, \cite{bourgain-guth} and the references therein), our proof does not rely on Fourier restriction theory thanks to the averaging over the dilation parameter $t$. In particular, the curvature of the cosphere $\Sigma_\rho = \{ \xi: \rho(\xi)=1 \}$ does not play a role in the argument (\emph{cf}. \cite{christ-sogge}, \cite{christ-sogge-survey}), which allows us to work with generalized Riesz means with respect to any smooth homogeneous distance function. 

\smallskip {\it Notation.}  Given two quantities $A$, $B$ we use the notation  $A\lc B$  to mean that there is a constant $C$ such that $A\le CB$. We use $A\approx B$ 
if $A\lc B$ and $B\lc A$.

\smallskip

{\it This paper.} In \S\ref{mainwtsect} we formulate Theorems \ref{H1thm} and \ref{Lpthm} on strong convergence for Riesz means of critical index in $\bbR^d$ and reduce their proof to the main weak type inequality stated in Theorem  \ref{mainthm}.
Some preliminary  estimates are contained in \S\ref{prelsect}.
The proof of the main Theorem \ref{mainthm} is given in \S\ref{proofofmainthm}.
 In \S\ref{transferencesect}
we use transference arguments to prove the positive results in Theorems \ref{H1thmTd} and \ref{LpthmTd}. In \S\ref{L1resultsect} we discuss a weaker result for $L^1$ functions.  In \S\ref{sharpnesssect} we show the essential sharpness of our $L^p$ results, namely that Theorems \ref{LpthmTd} and \ref{Lpthm} require the condition $q\le p'$
(the failure of the maximal theorems for $h^1$ already follows from Stein's example \cite{stein-cortona}).  In \S \ref{Hpsect} we include the proof of  an extension of a theorem by Stein, Taibleson and Weiss 
(\cite{stein-taibleson-weiss}), namely an $H^p\to L^{p,\infty}$ estimate for the maximal function $\sup_{t>0}|R_t^{\la(p)} f(x)|$ associated with generalized Riesz means in Hardy spaces $H^p$ with  $p<1$. Finally, we discuss some open problems in \S \ref{sec:open}.

\section{The main weak type estimate}\label{mainwtsect}
We  state results on $\bbR^d$ which are analogous to Theorems \ref{H1thmTd} and \ref{LpthmTd} and reduce them to a crucial inequality for a vector-valued operator stated in Theorem \ref{mainthm}.
 Let $\rho$ be as in the introduction. Recall the definition of Riesz means $R^\la_t$ for Fourier integrals from \eqref{rieszmeansRd}.

\begin{thm} \label{H1thm} Let  $q<\infty$ and  $\la(1)=\frac{d-1}{2}$. Then, for all $f\in H^1(\bbR^d)$,
for all $\alpha>0$,
\[
\meas\Big (\Big\{x\in \bbR^d: \sup_{T>0} \Big(\frac 1T \int_0^T|R^{\la(1)}_t \!f(x) |^q dt\Big)^{1/q} >\alpha\Big\}\Big) 
\le C\alpha^{-1}\|f\|_{H^1(\bbR^d)}.
\]
 

\end{thm}

\begin{thm}\label{Lpthm} Let $1<p<2$,
 $q<p'$ and $\la(p)=d(\frac 1p-\frac 12)-\frac 12$.  Then, for all $f\in L^p(\bbR^d)$, for all $\alpha>0$,
\[
\meas\Big(\Big\{x\in \bbR^d: \sup_{T>0} \Big(\frac 1T \int_0^T|R^{\la(p)} _t\! f(x) |^q dt\Big)^{1/q} >\alpha\Big\}\Big) 
\le C\alpha^{-p}\|f\|_{L^p(\bbR^d)}.
\]

\end{thm}

As a consequence of these estimates we obtain
 $$\lim_{T\to 0} \Big(\frac 1T \int_0^T|R^{\la(p)}_t \!f(x)-f(x)  |^q dt\Big)^{1/q} =0$$ for almost every $x\in \bbR^d$, for every $f\in L^p(\bbR^d)$ when  $1<p<2$  and 
  $f\in h^1(\bbR^d)$ or $H^1(\bbR^d)$ when $p=1$.
  
We start the reduction of Theorems \ref{H1thm} and \ref{Lpthm} to Theorem \ref{mainthm} by replacing the multipliers for the Riesz means $R_t^\lambda$ with similar multipliers supported away from the origin, see \eqref{Slatdef} below. 

\subsection{\it Contribution near the origin}\label{origincontr}
Let $\upsilon_0\in C^\infty(\bbR)$ so that $\upsilon_0(\rho)=1$ for $\rho\le 4/5$ and $\upsilon_0(\rho)=0$ for $\rho\ge 9/10$. It is then standard that the maximal function $\sup_{t>0} |\cF^{-1}[\upsilon_0(\rho(\cdot/t)) (1-\rho(\cdot/t))_+^\la \widehat f]$ defines an operator of weak type $(1,1)$ and bounded on $L^p$ for all $p>1$. A small complication occurs if $\rho$ is not sufficiently smooth at the origin. We address this complication as follows.

Define, for $N>0$,  
the functions $u$, $u_N$ with domain $(0,\infty)$ by 
$u(\tau)=\upsilon_0(\tau)(1-\tau)^\la$ and 
$u_N(s)=u(s^{1/N})$. It is then straightforward to check that for all $M$
\[\int_0^\infty s^{M}  |u_N^{(M+1)} (s)| ds <\infty \] and we have the subordination formula (\cite{trebels})
\Be\label{subord} u(\rho(\xi))= u_N(\rho^N(\xi))= \frac{(-1)^{M+1}}{M!}\int_0^\infty\big (1-\frac{(\rho(\xi))^N}{s}\big)^M_+ 
s^M u_N^{(M+1)} (s) ds\Ee
which is proved by integration by parts.
Given any $m>0$ one has   $|\cF^{-1} [(1-\rho^N)^M_+ ](x)| \lc_m (1+|x|)^{-m}$
provided $M$ and $N$ are large enough.
This is used to show that  $\sup_{t>0}| \cF^{-1}[ u\circ\rho(\cdot/t)\widehat f]$  is dominated by a constant times the Hardy-Littlewood maximal function of $f$ (see also Lemma \ref{originHardythm}).

We can now replace the operator $R^\la_t$ in Theorems \ref{H1thm} and \ref{Lpthm} by $S_t^\la $ defined by
\Be\label{Slatdef}\widehat {S^\la_t f}(\xi)= (1-\upsilon_0(\rho(\xi/t)))(1-\rho(\xi/t))_+^\la \widehat f(\xi).\Ee

\subsection{\it Further decompositions}\label{Littlewood-Paley} We first recall  standard dyadic decompositions on the frequency side.
Let $\eta\in C^\infty_c(\bbR^d\setminus \{0\})$ such that $\eta$ is nonnegative, 
 \Be \label{eta=1}\eta (\xi)=1 \text { on } \{\xi: \rho(\xi/t) \in[1/4,4], \,\,1/2\le t\le 2\}. \Ee
Define $\cL_k f$ by $\widehat {\cL_k f}(\xi)= \eta(2^{-k}\xi)\widehat f(\xi)$.

We use  the nontangential version of 
 the Peetre maximal operators 
$$\fM_k f(x)= \sup_{|h|\le 2^{-k+10}d}|\cL_k f(x+h)|$$ 
and the associated square function 
\Be \label{fS} \fS f(x)= \Big(\sum_{k\in \bbZ} |\fM_k f(x)|^2 \Big)^{1/2}.\Ee
Then 

\begin{subequations}
\Be\label{PeetreineqH1}\|\fS f\|_{L^1} \le C \|f\|_{H^1},\Ee
and 
\Be\label{PeetreineqLp}\|\fS f\|_{L^p} \le C_p \|f\|_{L^p} ,\text{ $1<p<\infty$. }\Ee
\end{subequations}
see  (Peetre \cite{peetre}).

The inequalities in Theorems \ref{H1thm} and \ref{Lpthm} follow from
$$
\Big\| \sup_{T>0} \Big(\frac 1T\int_0^T |S_t^{\la(p) }\! f|^q dt\Big)^{1/q} 
\Big\|_{L^{p,\infty}}
\lc \|\fS f\|_p
$$ for $1\le p<2$, $q<p'$. Here $L^{p,\infty}$ is the  weak type Lorentz space and the 
expression $\|g\|_{L^{p,\infty}} =\sup_{\alpha>0}\alpha (\meas(\{x:|g(x)|>\alpha\}))^{1/p}$ is the standard quasi-norm on $L^{p,\infty}$.
We may, by H\"older's inequality, 
assume that $2\le q<p'$. 
We can then use
\begin{equation}\label{ellqsum}
\sup_{T>0} \Big(\frac 1T\int_0^T |S^{\la(p) }_t \!f(x)|^q\,dt\Big)^{1/q} 
\le 2^{1/q} \Big(\sum_{k\in \bbZ}  2^{-k} \int_{2^{k} }^{2^{k+1}} |S_t^{\la(p)} \!f(x)|^qdt\Big)^{1/q}.
\end{equation}

We now use the standard idea to decompose the  multiplier $(1-\upsilon_0\circ\rho) (1-\rho)_+^\la $ into pieces supported where 
$\rho(\xi)\in [1-2^{-j}, 1-2^{-j-2}]$.
Generalizing slightly we assume that we are given  $C^\infty$ functions $\vphi_j$ supported in $[1-2^{-j}, 1-2^{-j-2}]$ and satisfying
\[
\|\partial^n \vphi_j\|_\infty \le C_n 2^{jn}.
\]
for $n=0,1,2,\dots$. Let $I:=[1,2]$. 
For $t\in I$, $k\in \bbZ$  define
\Be\label{Tjkdef}
\widehat {T_j^k f}(\xi,t)
=
\vphi_j(\rho(2^{-k}t^{-1}\xi))\widehat f(\xi)
\Ee
We may decompose $S^\la_{2^k t}  f =\sum_{j\ge 1} 2^{-j\la} T_j^k f(x,t)$,  with $T^j_k$ of the form in \eqref{Tjkdef}. The asserted estimates for  $S^{\la(p)}_t$ follow now from weak type bounds for the expression on the right hand side of \eqref{ellqsum}. 
By \eqref{eta=1} we have $\eta(2^{-k}\xi)=1$ whenever $\rho(2^{-k}\xi/t)\in \text{supp}(\vphi_j)$ for any $t\in I$.
Thus after changing variables the desired estimate can be restated  as
\[
\Big\|\Big(\sum_{k\in \bbZ}  \int_I 
\Big|\sum_{j=1}^\infty 
2^{-j\la(p)}T_{j}^{k} \cL_k f(\cdot,t)
\Big |^qdt\Big)^{1/q}\Big\|_{L^{p,\infty}} \lc \|\fS f\|_p.
\]
Since $\ell^2\subset \ell^q$ for $q\ge 2$ this follows from the following  stronger statement,
our main estimate.

\begin{thm}\label{mainthm}
For $1\le p<2$, $\la(p)= d(1/p-1/2)-1/2$, $q<p'$, 
\[
\Big\|\Big(\sum_{k\in \bbZ}  \Big(\int_I 
\Big|\sum_{j=1}^\infty 
2^{-j\la(p)}T_{j}^{k} \cL_k f(\cdot ,t)
\Big |^q dt\Big)^{2/q}\Big)^{1/2} \Big\|_{L^{p,\infty}(\R^d)} \lc \|\fS f\|_{L^p(\R^d)}.
\]
\end{thm}
The theorem  will be proved in \S\ref{proofofmainthm}. Some preparatory material is contained in \S\ref{prelsect}.

\section{Preliminary estimates}\label{prelsect}
We gather elementary estimates for the operators $T^k_j$ defined in \eqref{Tjkdef}.
\begin{lemma} \label{L2lemma}For $2\le q\le \infty$,
\[
\Big\|\Big(\int_{1}^2 |T_j^k f(\cdot, t)|^qdt \Big)^{1/q} \Big\|_2
\lc 2^{-j/q} \|f\|_2.\]
\end{lemma}
\begin{proof}
Use the convexity inequality, 
$\|\ga\|_q\le \|\gamma\|_2^{2/q}\|\gamma\|_\infty^{1-2/q}$,
for $\ga\in L^q([1, 2])$, and for $\ga\in C^1$ we have
$\|\ga\|_\infty \lc \|\ga\|_2^{1/2}(\|\ga\|_2+\|\ga'\|_2)^{1/2}$ and hence
\Be\label{2-q-bd}\Big(\int_{1}^2 |\gamma(t)|^qdt\Big)^{1/q}
 \lc \Big(\int_{1}^2 |\gamma(t)|^2dt\Big)^{1/2} +
\Big(\int_{1}^2 |\gamma(t)|^2dt\Big)^{\frac 12(\frac 12+\frac 1q)} 
\Big(\int_{1}^2 |\gamma'(t)|^2dt\Big)^{\frac 12(\frac 12-\frac 1q)} 
.\Ee
We obtain after some standard 
estimations
\begin{align*}
\Big\|\Big(\int_{1}^2 |T_j^k f(\cdot, t)|^2dt\Big)^{1/2}\Big\|_2 +2^{-j}\Big\|\Big(\int_{1}^2 |\frac{d}{dt} T_j^k 
f(\cdot, t)|^2dt\Big)^{1/2}\Big\|_2 
&\lc 2^{-j/2} \|f\|_2
\end{align*}
and then the assertion of the lemma follows from  \eqref{2-q-bd} applied to
 $\gamma(t)={T_j^k f}(x,t)$, followed by H\"older's inequality in $x$.
\end{proof}

To prove the $L^1$ 	estimate we rely on a  spherical decomposition introduced in \cite{cordoba}. For each fixed $j\geq 1$, we use a $C^\infty$ partition of unity $\{ \chi_{j,\nu} \}_{\nu \in \cZ_j}$ for an index set $\cZ_j$ with $\#\cZ_j=O(2^{j(d-1)/2})$, which has the following properties; each $\chi_{j,\nu}$ is homogeneous of degree $0$, the restriction of the support of $\chi_{j,\nu}$ to the sphere $\{\xi:|\xi|=1\}$ is supported in a set of diameter $2^{-j/2}$ and each unit vector is contained in the supports of $\chi_{j,\nu}$ for $O(1)$ indices $\nu$. We may choose the index set $\cZ_j$ so that for every $\nu$, there is a unit vector $\xi_{j,\nu}\in \supp{ \chi_{j,\nu}}$ so that $\dist(\xi_{j,\nu},\xi_{j,\nu'})\ge c 2^{-j/2}$ for $\nu\neq\nu'$. We assume that the $\chi_{j,\nu}$ satisfy the  natural differential estimates, i.e.
$\partial_\xi^{\beta} \chi_{j,\nu}(\xi) = O(2^{\frac{j}{2}(\beta_1+\dots\beta_d)})$.
Define $T^k_{j,\nu}$ by
\Be\label{Tjknudef}
\widehat {T_{j,\nu}^{k} f}(\xi,t)
=
\chi_{j,\nu}(\xi)\vphi_j(\rho(2^{-k}t^{-1}\xi))\widehat f.
\Ee
Let 
$K_j= \cF^{-1}[ 
\vphi_j(\rho(\cdot))],$ and $K_{j,\nu}= \cF^{-1}[ 
\vphi_j(\rho(\cdot))\chi_{j,\nu}].$
Let 
$\Phi_0\in C^\infty_c(\bbR^d)$ supported in $\{x:|x|\le 1\}$ so that $\Phi_0(x)=1$ 
for $|x|\le 1/2$ and, for $n\ge 1$,  let $\Phi_n(x)= \Phi_0(2^{-n}x)-\Phi_0(2^{1-n}x)$.
Define, for $n=0,1,2,\dots$,
\begin{align*}
K_{j}^n(x) &= K_{j}(x)\Phi_n(2^{-j}x)
\\K_{j,\nu}^n(x) &= K_{j,\nu}(x)\Phi_n(2^{-j}x)
\end{align*}
and
\begin{align*}
T_{j}^{n,k} f(x,t)  &= (2^{k} t)^d K_{j}^n(2^kt\cdot) * f,
\\
T_{j,\nu}^{n,k} f(x,t)  &= (2^{k} t)^d K_{j,\nu}^n(2^kt\cdot) * f.
\end{align*}
Then 
\Be \label{decompositions}T_j^k f=\sum_{\nu\in \cZ_j} T_{j,\nu}^kf =\sum_{n=0}^\infty T_j^{n,k}\!f
=\sum_{n=0}^\infty 
\sum_{\nu\in \cZ_j} T_{j,\nu}^{n,k}  \!f.\Ee

\begin{lemma} \label{multsize}
Let $\Sigma_\rho=\{\xi:\rho(\xi)=1\}$. Then 
$$|\widehat  {K_j^n}(\xi)|
\le C_{M_0,M_1} 2^{-nM_0} (1+2^j\text{\rm dist}(\xi,\Sigma_\rho))^{-M_1}.$$
\end{lemma}
\begin{proof} [Sketch of Proof]
Let $\Psi(x)=\Phi_0(x/2)-\Phi_0(x)$. Then, for $n\ge 1$, we may use that $\widehat \Psi$ has  vanishing moments   and  write 
\Be \label{taylor} \begin{aligned}
\widehat  {K_j^n}(\xi) &=\int \varphi_j(\rho(\xi-y)) 2^{(j+n)d} \widehat \Psi(2^{j+n} y) dy
\\
&=\int_0^1\frac{(1-s)^{N-1}}{(N-1)!} \int  \inn y\nabla^N \big[\varphi_j\circ\rho]  (\xi-sy)  2^{(j+n)d} \widehat \Psi(2^{j+n} y) dy\, ds
\end{aligned}
\Ee
by Taylor's formula. The estimate is now straightforward.
When $n=0$ we just use the first line in \eqref{taylor} with $\Psi$ replaced by $\Phi_0$.
\end{proof}

For each $\nu$ choose $\xi_{j,\nu}$ such that $\rho(\xi_{j,\nu})=1$ and $\xi_{j,\nu}\in \text{supp}( \chi_{j,\nu})$. Take $e_{j,\nu} =\frac{\nabla\rho(\xi_{j,\nu})}{|\nabla\rho(\xi_{j,\nu})|}$ and let $P_{j,\nu}$ be the orthogonal projection to $e_{j,\nu}^\perp$, i.e.
\Be\label{projection} P_{j,\nu} h=h-\inn{h}{e_{j,\nu}}e_{j,\nu}. \Ee

\begin{lemma} \label{intbypartsest}
For every $M\ge 0$,
\Be\label{Kjnukernelestimate}\sup_{t\in I} |t^dK_{j,\nu}(tx)|\le C(M) \frac{2^{-j(d+1)/2}}
{(1+2^{-j} |x|)^{M} (1+2^{-j/2}|P_{j,\nu}(x)|)^M}.
\Ee
\end{lemma} 
\begin{proof} 
This  is standard (and follows after integration by parts), see, e.g.,
\cite{cordoba}, \cite{christ-sogge-survey}, or \cite{seeger-archiv}.
\end{proof}

\begin{lemma}\label{Tnklemma}
(i) For $k\in \bbZ$, 
$$\big\|\sup_{t\in I} |T^{n,k}_{j} f(\cdot,t)|\,\big \|_1 \le C_N 2^{j\frac{d-1}{2}} 2^{-nN} \|f\|_1.$$

(ii) For $1< p\le 2$, $q\le p'$ and $k\in \bbZ$,
$$\Big \|\Big(\int_{ I} |T^{n,k}_{j}f(\cdot,t)|^q dt\Big)^{1/q} \Big \|_p \le C_N 2^{j(d(\frac 1p-\frac 12)-\frac 12)} 2^{-nN}  \|f\|_p.$$

(iii) For $2\le q\le \infty$,
\[
\Big\|\Big(\int_{1}^2 |T_j^{n,k} f(\cdot, t)|^qdt \Big)^{1/q} \Big\|_2
\lc 2^{-j/q} 2^{-nN} \|f\|_2.\]
\end{lemma}

\begin{proof}

Lemma \ref{intbypartsest}  easily implies
$\|\sup_{t\in I} |T^{n,k}_{j,\nu} f(\cdot,t)|\,\|_1 \le C_N 2^{-nN} \|f\|_1$ and 
part (i) follows after summing in $\nu$. 
Using Lemma  \ref{multsize} we see that the proof of Lemma \ref{L2lemma} also gives
\begin{align*}
\Big\|\Big(\int_{1}^2 |T_j^{n,k} f(\cdot, t)|^2 dt\Big)^{1/2}\Big\|_2 \lc_N 2^{-nN} 2^{-j/2} \|f\|_2.
\end{align*}
Part (ii) now follows by complex interpolation.

Part (iii) for $q=2$ is just the previous displayed inequality. For $q>2$ it   follows by the argument in Lemma \ref{L2lemma} (\cf. \eqref{2-q-bd}) applied to $T^{n,k}_j$ in place of $T^k_j$,
in conjunction with Lemma \ref{multsize}.
\end{proof}

\section{Proof of Theorem \ref{mainthm}} \label{proofofmainthm}
The proof combines ideas that were used in the proof of weak type inequalities for 
Bochner-Riesz means and other radial multipliers, 
 and elsewhere  (\cite{fefferman}, \cite {christ-wt}, \cite{christ-pp}, 
\cite{christ-sogge},
\cite{seeger-indiana}). It combines atomic decompositions with Calder\'on-Zygmund estimates
using $L^r$-bounds  for $r>p$ in the complement of suitable exceptional sets together with
analytic interpolation arguments inspired by \cite{christ-pp}.

In this section we  fix  a Schwartz function $f$ whose Fourier transform has compact support in $\bbR^d\setminus \{0\}$.
Observe that  then $\cL_kf=0$ for all but a finite number of indices $k$ (depending on $f$). This assumption together with the Schwartz bounds can be used to justify the a priori finiteness of various expressions showing up in the arguments  below, but they do not enter quantitatively in the estimates.

We  need to prove the inequality
\Be\label{mainthmineq}
\meas\Big\{  x\in \bbR^d: \Big(\sum_k\Big[\int_I\Big|\sum_{j=1}^\infty 
2^{-j\la(p)}T_{j}^{k} \cL_k  f(x,t)
\Big|^q dt\Big]^{2/q}\Big)^{1/2} >\alpha\Big\} 
\lc \alpha^{-p}\|\fS f\|_p^p,\Ee
for arbitrary but fixed $\alpha>0$. The implicit constant does not depend on $\alpha$ and the choice of $f$.

\subsection{\it Preliminaries on atomic decompositions} \label{preliminaries}
Let $\cR_k$ be the set of dyadic cubes  of side  length  $2^{-k}$ so that each  $R\in \cR_k$ is of the form
$\prod_{i=1}^d[n_i2^{-k}, (n_i+1)2^{-k})$ for some $n\in \bbZ^d$.
For $\mu\in \bbZ$ let  $$\Om_\mu= \{x: |\fS f(x)|>2^{\mu}\}$$
and let $\cR_k^\mu$ be the set of dyadic  cubes  of length $2^{-k}$ with the property that 
$$|R\cap\Omega_\mu|\ge |R|/2 \text{ and } |R\cap \Om_{\mu+1}|< |R| /2.$$
Clearly if $\fS f\in L^p$ then every dyadic cube in $\cR_k$ belongs to exactly one of the sets 
$\cR^\mu_k$.
We then have (\cite{crf})
\Be\label{basicL2} 
\sum_{k\in \bbZ} \sum_{R\in \cR^\mu_k} \int_R|\cL_k f|^2dx  \lc 2^{2\mu} \meas (\Om_\mu).
\Ee
For completeness we give the argument. Observe that
\[|\cL_kf(x)|\le \fM_k f(z), \quad\text{for  $x, z\in R$, $R\in \cR^\mu_k$}.\] Let 
$$\widetilde \Om_\mu= \{x: M_{HL} \bbone_{\Om_\mu} > 10^{-d}\},$$  where $M_{HL}$ denotes  the Hardy-Littlewood maximal operator. Then  $$\meas(\widetilde\Om_\mu)\lc \meas (\Om_\mu)$$ and we have $\cup_k  \cup_{R\in \cR^\mu_k} R\subset \widetilde\Om_\mu$.
Now 
\begin{align*}
&\sum_{k\in \bbZ} \sum_{R\in \cR^\mu_k} \|\bbone_R \cL_k f\|_2^2
\le \sum_{k\in \bbZ} \sum_{R\in \cR^\mu_k}2  \int_{R\setminus \Om_{\mu+1}}|\fM_k f(x)|^2 dx
\\ 
&\le 2 \int_{\widetilde\Om_\mu\setminus \Om_{\mu+1}} \sum_{k\in \bbZ} 
|\fM_k f(x)|^2 dx
\le 2^{2\mu+1} \meas(\widetilde \Om_\mu) \le C 2^{2\mu} \meas(\Om_\mu)
\end{align*} which yields  \eqref{basicL2}.

Next we work with a Whitney decomposition of the open set $\widetilde \Om_\mu$, 
which is  a disjoint union of dyadic cubes $W$, such that
$$\diam (W) \le \dist(W, \widetilde \Om_\mu^\complement) \le 4 \diam (W).$$
See 
\cite[ch. VI.1]{stein-book}.
We denote by $\fW^\mu$ the collection of these Whitney cubes.
Each $R\in \cR_k^\mu$ is contained in a unique $W(R)\in \fW^\mu$.
For each  $W$ define
\Be \label{Rkmw}
\cR_k^{\mu}(W)= \{R\in \cR_k^\mu: \, R\subset W\}
\Ee and
$$\ga_{W,\mu} = \Big(\frac{1}{|W|} \sum_{k}\sum_{\substack {R\in \cR_k^{\mu}(W)}}\int_R|\cL_k f(y)|^2 dy\Big)^{1/2}.
$$
Define 
\Be\label{Udef}U(x)= \sum_\mu\sum_{W\in \fW^\mu} \ga_{W,\mu}^p \bbone_W(x).
\Ee
Observe that
\Be\label{Uestderivation}
\begin{aligned} 
\|U\|_1&= \sum_\mu\sum_{W\in \fW^\mu} |W|\ga_{W,\mu}^p  
=\sum_\mu\sum_{W\in \fW^\mu} |W|^{1-p/2} (|W|^{1/2}\ga_{\mu,W})^{p}
\\
&\le \sum_\mu 
\Big(\sum_{W\in \fW^\mu}|W|\Big)^{1-p/2} 
\Big(\sum_{W\in \fW^\mu} |W| \ga_{\mu,W}^2\Big)^{p/2}
\\
&\le \sum_\mu |\widetilde\Om_\mu|^{1-p/2}  
 \Big( \sum_{k}\sum_{W\in \fW^\mu}\sum_{\substack {R\in \cR_k^{\mu}(W)}}\|\bbone_R\cL_k f\|_2^2\Big)^{p/2}
\\
&\lc \sum_\mu |\Om_\mu|^{1-p/2}  (2^{2\mu} |\Omega_\mu|)^{p/2} 
\lc \sum_{\mu}2^{\mu p} |\Om_\mu| ,
\end{aligned}\Ee
by \eqref{basicL2}, and thus
\begin {equation} \label{Uest} 
\sum_\mu\sum_{W\in \fW^\mu} |W|\ga_{W,\mu}^p  =\|U\|_1\lc \|\fS f\|_p^p.
\end{equation}

For $\alpha>0$
let
\Be\label{Odef}
\cO_\alpha= \{x: M_{HL} U > \alpha^p\}
\Ee
and 
\Be\label{Otildedef}
\widetilde \cO_\alpha= \{x: M_{HL} \bbone_{\cO_\alpha}(x) > (10d)^{-d}\}
\Ee
so that $\cO_\alpha\subset \widetilde \cO_\alpha$ and 
\Be \label{measOa}
\meas(\widetilde\cO_\alpha) \lc 
\meas(\cO_\alpha) \lc 
\alpha^{-p}\|\fS f\|_p^p\,.
\Ee

Let $\fQ_\alpha=\{Q\}$ be the collection of Whitney cubes for the set $\widetilde \cO_\alpha$
(\cf. \cite[ch. VI]{stein-book})
so that
$$\diam (Q) \le \dist(Q, \widetilde \cO_\alpha^\complement) \le 4 \diam (Q).$$

In analogy to the usual terminology of ``good'' and ``bad" functions in Calder\'{o}n-Zygmund theory we split,  for fixed $\alpha$,  the collection ${\fW}^\mu$ into two subcollections $\Wgood\equiv \Wgood(\alpha) $ and
$\Wbad\equiv \Wbad(\alpha)$ by setting
\Be\label{Wgoodbad}  \begin{aligned}
\Wbad&
 \,= \,\big\{ W\in \fW^\mu :  \ga_{W,\mu}>\alpha\big\},
\\
\Wgood&\, =\,\,\big\{ W\in \fW^\mu :  \ga_{W,\mu}\le \alpha\big\}.
\end{aligned}\Ee

We relate the collection $\Wbad$ with the collection of Whitney cubes $\fQ_\alpha$ for the set $\widetilde \cO_\alpha$.

\begin{lemma} \label{fixedWQcontainment}
Let $W \in \Wbad$. Then $W \subset \cO_\alpha$. Moreover, there is a unique cube $Q=Q(W)\in \fQ_\alpha$ containing $W$.
\end{lemma}
\begin{proof}
For the first statement, assume otherwise that there is $x \in W\cap \cO_\alpha^\complement$ for some $W \in \Wbad$. Then $U(x)\le \alpha^p$ and therefore $\ga_{W,\mu}^p\le\alpha^p$, which is a contradiction.

For the second statement, we first claim that $W^*\subset \widetilde \cO_\alpha$, where $W^*$ is the $10 d^{1/2}$-dilate of $W$ (with same center). The claim follows because for all $y\in W^*$ we have $M_{HL} \bbone_{\cO_\alpha}(y)\ge |W|/|W^*|= (10\sqrt d)^{-d}$ by the first statement.
Let $x_W$ be the center of $W$. Then by the claim
\[ 
\dist(x_W, (\widetilde \cO_\alpha)^\complement)\ge 
\dist(x_W, (W^*)^{\complement})= \frac{\diam(W^*)}{2\sqrt d}= 5\diam (W).
\]


Let $Q\in \fQ_\alpha$ such that $x_W\in Q$. Then  the last displayed inequality implies
\begin{align*}
5  \diam (W) \le \dist(x_W, (\widetilde \cO_\alpha)^\complement) 
\le 
\diam(Q)+ \dist(Q, (\widetilde \cO_\alpha)^\complement) 
\le 
5\diam(Q)
\end{align*}
and hence $\diam(Q)\ge  \diam(W)$. Since both $W$, $Q$ are dyadic cubes containing $x_W$ this implies $W\subset Q$. Uniqueness of $Q$ follows since the  cubes in $\fQ_\alpha$ have disjoint interior.
\end{proof}

In light of Lemma \ref{fixedWQcontainment}, we also set, for a dyadic cube $Q \in \fQ_\alpha$,
\Be\label{WmuQ}
\fW^\mu(Q)= \{ W\in \fW^\mu_{\text{bad}}:\, W\subset Q\}.
\Ee

\begin{lemma} \label{QaverageWsum} Let $Q\in \fQ_\alpha$.
Then
\[
\sum_\mu \sum_{\substack{W\in \fW^\mu(Q)}} |W|\ga^p_{W,\mu} \le 10^d \alpha^p|Q|.
\]
\end{lemma}
\begin{proof}
Since $Q$ is a Whitney cube for the set $\widetilde \cO_\alpha$, there is $\widetilde x\in\widetilde \cO_\alpha^\complement\subset \cO_\alpha^\complement$ such that $\dist(\widetilde x, Q)\le 4\diam(Q)$. 
If $Q_*$ denotes  the cube centered at $\widetilde x$ with diameter equal to $10\diam (Q)$ then $Q\subset Q_*$.
Since $\widetilde x \in \cO_\alpha^\complement$ we have $M_{HL} U(\widetilde x) \le \alpha^p$.
Hence $\int_Q U\le \int_{Q_*} U\le \alpha^p|Q_*|= 10^d\alpha^p |Q|$ and the assertion follows.
\end{proof}

\subsection{\it Outline of the proof of the weak type inequalities}\label{outlinesect}
For $R\in\cR_k$ let 
\Be\label{eR} e_R(x) =\bbone_R (x)  \cL_k f(x)\Ee and as in 
\eqref{Wgoodbad} split $\cL_k f=g^k+b^k$, where
\begin{align}
g^k&=\sum_\mu\sum_{\substack{W\in \Wgood}}
\sum_{\substack{R\in \Rkmw}}e_R,
\\
b^k&= 
\sum_\mu\sum_{W\in \Wbad}
\sum_{\substack{R\in \Rkmw}}e_R.
\end{align}
\medskip

In view of \eqref{measOa} it suffices to  show, for $2\le q<\infty$,
\Be\label{goodfctest}
\meas\Big\{  x: \Big(\sum_k \Big(\int_I \Big|\sum_{j=1}^\infty 2^{-j\la(p)}T_j^k g^k(x,t)\Big|^q  dt\Big)^{2/q} \Big)^{1/2}>\alpha/2\Big\} 
\lc \alpha^{-p}\|\fS f\|_p^p
\Ee
and
\Be\label{badfctest}
\meas\Big\{  x\in \widetilde \cO^\complement_\alpha: \Big(\sum_k \Big(\int_I \Big|\sum_{j=1}^\infty 2^{-j\la(p)}T_j ^k b^k(x,t)\Big|^q  dt \Big)^{2/q}\Big)^{1/2}>\alpha/2\Big\} 
\lc \alpha^{-p}\|\fS f\|_p^p\,.
\Ee
Since $\lambda(p)\ge 1/p-1>-1/q$  we can use Lemma \ref{L2lemma} to bound 
\begin{align*}
&\Big\|
\Big(\sum_k \Big(\int_I \Big|\sum_{j=1}^\infty 2^{-j\la(p)}T_j^k g^k(x)\Big|^q  dt \Big)^{2/q}\Big)^{1/2}\Big\|_2
\\&\lc \sum_{j=1}^\infty 2^{-j\la(p)} \Big\|
\Big(\sum_k \Big(\int_I \big|T_j^k g^k(x)\big|^q  dt \Big)^{2/q}\Big)^{1/2}\Big\|_2
\\
&\lc 
 \sum_{j=1}^\infty 2^{-j(\la(p)+\frac 1q)} 
\Big(\sum_k \|g^k\|_2^2\Big)^{1/2}\lc 
\Big(\sum_k \|g^k\|_2^2\Big)^{1/2}.
\end{align*}
Hence, by 
 Tshebyshev's inequality,  the left hand side of \eqref{goodfctest} is bounded by
\[
4\alpha^{-2} \Big\|
\Big(\sum_k \Big(\int_I \Big|\sum_{j=1}^\infty 2^{-j\la(p)}T_j^k g^k(x)\Big|^q  dt \Big)^{2/q}\Big)^{1/2}\Big\|_2^2
\lc \alpha^{-2} \sum_k \|g^k\|_2^2\,.
\]
Now 
\begin{align*}
 &\sum_k \|g^k\|_2^2 =\sum_k \Big\|
 \sum_\mu\sum_{\substack{W\in \Wgood}}
\sum_{\substack{R\in \Rkmw}}e_R\Big\|_2^2
\le 
\sum_k 
 \sum_\mu\sum_{W\in \Wgood}
\sum_{\substack{R\in \Rkmw}}\|e_R\|_2^2
\\
&=  \sum_\mu
\sum_{\substack{W\in \Wgood}} |W| \gamma_{W,\mu}^2
\lc \alpha^{2-p} 
\sum_\mu
\sum_{\substack{W\in \fW^\mu}} |W| \gamma_{W,\mu}^p  \lc \alpha^{2-p} \|\fS f\|_p^p,
\end{align*}
where we have used $\gamma_{W,\mu} \leq \alpha$ for $W\in \Wgood$.  \eqref{goodfctest} follows.

We turn to 
\eqref{badfctest}.
We write $L(Q)=m$ if the side length of $Q$ is $2^m$. 
Define, for $m\ge -k$, 
\Be\label{Bkm}
B^k_m= \sum_{\substack {Q\in \fQ_\alpha\\ L(Q)=m}}
\sum_\mu\sum_{\substack{W\in \Wmq}}
\sum_{\substack{R\in\Rkmw}}e_R
\Ee
so that $b^k=\sum_{m\ge -k} B^k_m$.

Note that  for $R\in \cR_k^{\mu}(W)$, 
we have $L(W)\ge -k$.
Then $$b^k=\sum_{m\ge -k} B^k_m=\sum_{m\ge -k} \sum_{\sigma\ge 0} B^k_{m,\sigma},$$ where
\Be\label{Bkmka}
B^k_{m,\sigma}= \sum_{\substack {Q\in \fQ_\alpha\\ L(Q)=m}}
\sum_\mu\sum_{\substack{W\in \Wmq\\ L(W)=-k+\sigma}}
\sum_{\substack{R\in \Rkmw}}e_R.
\Ee

We  handle the case of the contributions $T^k_j B^k_{m,\sigma}$ with $m\le j-k$ differently from those with $m> j-k$. Moreover we distinguish the cases where  $|j-k-m|\ge \sigma$ 
and $|j-k-m|< \sigma$. 
If we use Tshebyshev's inequality and take into account \eqref{measOa} we see that in order to establish  \eqref{badfctest} it suffices to show the following three inequalities, assuming  $2\le q<p'$
(and hence $p<q'\le 2$):
\Be\label{Lraway}
\Big\|\Big(\sum_k \Big[\int_I \Big|\sum_{j=1}^\infty 2^{-j\la(p)}
\sum_{\substack {(m,\sigma):m\le j-k,\\0\le \sigma\le j-m-k}}
T_j ^k 
B^k_{m,\sigma}(\cdot,t)\Big|^q  dt \Big]^{q'/q}\Big)^{1/q'}\Big\|_{L^{q'}(\bbR^d)}^{q'}\lc \alpha^{q'-p} \|\fS f\|_p^p,\,
\Ee
\Be\label{Lpaway}
\Big\|\Big(\sum_k \Big[ \int_I \Big|\sum_{j=1}^\infty 2^{-j\la(p)}
\sum_{\substack {(m,\sigma):m> j-k,\\0\le \sigma\le m+k-j}}
T_j ^k B^k_{m,\sigma}(\cdot,t)\Big|^q  dt \Big]^{p/q}\Big)^{1/p}\Big\|_{L^p(\bbR^d\setminus \widetilde \cO_\alpha)}^p 
\lc \|\fS f\|_p^p,
\Ee
and
\Be\label{Lpatomic}
\Big\|\Big(\sum_k \Big[\int_I \Big|\sum_{j=1}^\infty 2^{-j\la(p)}
\sum_{\substack{m, \sigma:\\ \sigma>| m+k-j|}}T_j ^k B^k_{m,\sigma}(\cdot,t)\Big|^q  dt \Big]^{2/q}
\Big)^{1/2}
\Big\|_{L^p(\bbR^d\setminus \widetilde \cO_\alpha)}^p 
\lc \|\fS f\|_p^p.
\Ee
The proofs will be given in Sections
\ref{Lrawaysect}, \ref{Lpawaysect}, and \ref{Lpatomicsect}.
We shall handle the cases $p=1$, $2\le q<\infty$, and  $1<p<2$, $2\le q<p'$, in a unified way but will need an additional analytic families interpolation argument for $1<p<2$.

\subsection{\it Analytic families}
Fix $p$, $\alpha$ and consider for $0\le \Re(z)\le 1$ the family of functions
	\Be\label{bcomplex}
	b^{k,z}_{Q,\sigma}= 
\sum_\mu\sum_{\substack{W\in \Wmq\\L(W)=-k+\sigma}}\gamma_{W,\mu,z}
\sum_{\substack{R\in \Rkmw}}e_R,
\Ee
 where for $W\in \Wbad$
\[ \gamma_{W,\mu,z}=   
\gamma_{W,\mu}^{p(1-z/2)-1}
\]
and $Q$ belongs to $\fQ_\alpha$. Observe  that $b^{k,z}_{Q,\sigma}$ is supported in $Q$.
Notice that $z\mapsto  \gamma_{W,\mu,z}$ is an entire function for   $W\in \Wbad$.
We also set 
\Be \label{Bm-complex} 
B^{k,z}_{m,\sigma} = \sum_{\substack {Q\in \fQ_\alpha\\L(Q)=m}} b^{k,z}_{Q,\sigma} 
\Ee
and, for $0\le \Re(z)\le 1$, define $p_z$ and $\la(p_z)$ by 
\Be\label{laz} \frac{1}{p_z}=1-z+\frac z2, \quad \lambda(p_z)= \frac{d(1-z)-1}{2}.\Ee
If $1<p<2$ then we set  $\vth=2-2/p$ so that
\[ p_\vth=p, \quad \lambda(p_\vth)=d(\frac 1p-\frac 12)-\frac 12,
\quad B^{k,\vartheta}_{m,\sigma}= B^k_{m,\sigma}.
\]
For $\Re(z)=1$ we have
\begin{lemma}\label{L2estimatesBsigma}
For fixed $k$, $m\ge -k$ let $\cN_{k,m} \subset \bbZ$.
Then
\[ \sum_{k\in \bbZ}\sum_{m\ge -k} \Big\|\sum_{\sigma\in \cN_{k,m}} B^{k,z}_{m,\sigma}\Big\|_2^2\lc \|\fS f\|_p^p, \quad \Re(z)=1.
\]
\end{lemma}
\begin{proof}
The left hand side  is equal to 
\begin{align*}
&\sum_k\sum_{m\ge -k} \Big\|\sum_{\substack {Q\in \fQ_\alpha\\ L(Q)=m}}
\sum_\mu\sum_{\sigma\in \cN_{k,m}} \sum_{\substack{W\in \Wmq\\ L(W)=-k+\sigma} }
\gamma_{W,\mu} ^{p(1-z/2)-1}
\sum_{\substack{R\in \cR^{\mu}_k(W)}}e_R\Big\|_2^2.
\end{align*}
Let for each $W$, $Q(W)$ be the unique cube in $\fQ_\alpha$  such that $W\subset Q$. We use that  for fixed $k$ the supports of the functions $e_R$, $R\in \cR_k$ have disjoint  interior and dominate for $\Re(z)=1$ the last display 
  by 
\begin{align*}
&
\sum_k\sum_{m\ge -k} \sum_{\substack {Q\in \fQ_\alpha\\ L(Q)=m}}
\sum_\mu\sum_{\substack{W\in \Wmq:\\L(W)+k\in \cN_{k,m}} }
\gamma_{W,\mu} ^{(\frac p2 -1)2}
\sum_{\substack{R\in \cR^\mu_k(W)}}\|e_R\|_2^2
\\
&\lc \sum_\mu\sum_{W\in \fW^\mu_{\text{bad}}} \gamma_{W,\mu}^{p-2}
\sum_{\substack{k:L(W)+k\in\\ \cN_{k, L(Q(W))} }}\sum_{R\in \cR^\mu_k(W)} \|e_R\|_2^2
\le \sum_{\mu}\sum_{W\in \fW^{\mu}} \gamma_{W,\mu}^p|W|
\lc \|\fS f\|_p^p.\qedhere
\end{align*}
\end{proof}

For $\Re(z)=0$ we have
\begin{lemma}\label{Qaverage}
There exists a universal constant $C$ 
dependent only on the dimension such that  for every $Q\in \fQ_\alpha$ and every 
$\cN\subset \bbN\cup \{0\}$ 
\[
 \int\Big(\sum_k\Big |\sum_{\sigma\in \cN} b_{Q,\sigma}^{k,z}(x)\Big|^2\Big)^{1/2} dx \le C \alpha^p|Q|, \,\, \text{ if }\Re(z)=0.
\]
\end{lemma}

\begin{proof}
For each $W\in \fW^\mu$ let $W_*$ be its double.
By Minkowski's inequality the left hand side is dominated by
\begin{align*} 
&\sum_\mu\sum_{\substack {W\in \fW^\mu\\W\subset Q}}\ga_{W,\mu}^{p-1}
\int\Big(\sum_{\substack {k:\\k+L(W)\in \cN}}\Big |\sum_{\substack{R\in \cR^{\mu}_k(W)}}e_R(x) \Big|^2\Big)^{1/2} dx 
\\
&\lc 
\sum_\mu\sum_{\substack {W\in \fW^\mu\\W\subset Q}}\ga_{W,\mu}^{p-1}
\int_{W_*}\Big(\sum_{k}\sum_{\substack{R\in \cR^{\mu}_k(W)}}|e_R(x) |^2\Big)^{1/2} dx
\end{align*}
which by the Cauchy-Schwarz inequality  can be estimated by 
\begin{align*} 
&\sum_\mu\sum_{\substack {W\in \fW^\mu\\W\subset Q}}\ga_{W,\mu}^{p-1}\Big(\sum_{k}\sum_{\substack{R\in \Rkmw}}\|e_R \|^2\Big)^{1/2} |W_*|^{1/2}
\lc 
\sum_\mu\sum_{\substack {W\in \fW^\mu\\W\subset Q}}|W^*|\ga_{W,\mu}^{p}
\lc \alpha^p|Q|.
\end{align*}
Here we have used Lemma \ref{QaverageWsum}.
\end{proof}

\subsection{\it Proof of  \eqref{Lraway}}\label{Lrawaysect}
Let  $1\le p<2$ and $2\le q<p'$. The asserted inequality follows from
\begin{subequations}
\begin{multline}\label{Lqprime-away-1}
\Big\|\Big(\sum_k\Big( \int_I \Big|\sum_{j\ge 2s} \sum_{0\le \sigma\le s}
2^{-j\la(p)}T_j ^k 
B^k_{j-k-s,\sigma}(\cdot,t)\Big|^q  dt \Big)^{q'/q}\Big)^{1/q'}\Big\|_{q'}\\
\lc (1+s)^{1-\frac 2q}
2^{-s(d-1)(\frac 1p-\frac 1{q'})}\alpha^{p(\frac 1p-\frac{1}{q'})} \|\fS f\|_p^{p/q'}\,
\end{multline}
and, 
\begin{multline}\label{Lqprime-away-2}
\Big\|\Big(\sum_k \Big(\int_I \Big| \sum_{0\le \sigma\le s}2^{-j\la(p)}T_j ^k 
B^k_{j-k-s,\sigma}(\cdot,t)\Big|^q  dt \Big)^{q'/q}\Big)^{1/q'}\Big\|_{q'}\\ \lc (1+j)^{1-\frac 2q}
2^{-j\frac{d-1}{2}(\frac 1p-\frac 1{q'})}\alpha^{p(\frac 1p-\frac{1}{q'})} \|\fS f\|_p^{p/q'},\quad \frac j2\le s\le j\,.
\end{multline}
\end{subequations}

If in addition $p>1$ we use a complex interpolation argument, embedding $B^k_{m,\sigma}$ in an analytic family of functions, see \eqref{Bm-complex}.

Define $r$ by \Be \label{rdefinition}
 \frac 1r= \Big(\frac 1p-\frac 1q\Big)\Big/ \Big(\frac 2p-1\Big),
\Ee  so that $1<r\le 2$ and for $\vth=2-2/p$ we have
$(1-\vth)(1,\frac 1r)+\vth (\frac 12, \frac 12)=(\frac 1p,\frac 1{q'}).$
Then by complex interpolation (i.e. the three lines lemma and duality) we deduce
 \eqref{Lqprime-away-1}, \eqref{Lqprime-away-2} 
  from
\begin{subequations}
\Be\label{L2away-sv1-z=1}
\Big\|\Big(\sum_k \int_I \Big|\sum_{j\ge 2s} \sum_{0\le \sigma\le s}
2^{-j\la(p_z)}T_j ^k 
B^{k,z}_{j-k-s,\sigma}(\cdot,t)\Big|^{2}  dt \Big)^{1/2}\Big\|_{2}\lc 
\|\fS f\|_p^{p/2},\, \\ \Re(z)=1,
\Ee
\Be\label{L2away-sv2-z=1}
\Big\|\Big(\sum_k \int_I \Big|  \sum_{0\le \sigma\le s}2^{-j\la(p_z)}T_j ^k 
B^{k, z}_{j-k-s,\sigma}(\cdot,t)\Big|^{2}  dt \Big)^{1/2}\Big\|_{2}\lc 
 \|\fS f\|_p^{p/2},\,\,\, \frac j2\le s\le j,\,\,\, \Re(z)=1.
\Ee
\end{subequations}
and
\begin{subequations}
\begin{multline}\label{Lqprime-away-1-z=0}
\Big\|\Big(\sum_k \int_I \Big|\sum_{j\ge 2s}  \sum_{0\le \sigma\le s}2^{-j\la(p_z)}T_j ^k 
B^{k,z}_{j-k-s,\sigma}(\cdot,t)\Big|^{r'}  dt \Big)^{1/r'}\Big\|_{r}\\ \lc (1+s)^{\frac 2r -1}
2^{-s(d-1)(1-\frac 1r)}\alpha^{p(1-\frac{1}{r})} \|\fS f\|_p^{p/r},\, \,\Re(z)=0,
\end{multline}
\begin{multline}\label{Lqprime-away-2-z=0}
\Big\|\Big(\sum_k \int_I \Big| \sum_{0\le\sigma\le s} 2^{-j\la(p_z)}T_j ^k 
B^{k, z}_{j-k-s,\sigma}(\cdot,t)\Big|^{r'}  dt \Big)^{1/r'}\Big\|_{r}\\
\lc (1+j)^{\frac 2r -1}
2^{-j\frac{d-1}{2}(1-\frac 1r)}\alpha^{p(1-\frac{1}{r})} \|\fS f\|_p^{p/r}, \quad \frac j2\le s\le j,\,\, \,\Re(z)=0.
\end{multline}
\end{subequations}
We note that for the special case $p=1$  inequalities
\eqref{Lqprime-away-1-z=0},
\eqref{Lqprime-away-2-z=0} with $r=q'$ and $z=0$ imply inequalities \eqref{Lqprime-away-1}, \eqref{Lqprime-away-2} with $p=1$.

The proof of 
\eqref{L2away-sv1-z=1}, 
\eqref{L2away-sv2-z=1}
 is straightforward, using orthogonality, i.e. 
the fact that for each $k$,  $t$, $\xi$ there are at most five $j$ for which  $\vphi_j(\rho(2^{-k}t^{-1}\xi))\neq 0$. Therefore 
we get for $\Re (z)=1$ (and $\Re(\la(p_z))=-1/2$)
\begin{align*}
&\Big\|\Big(\sum_k \int_I \Big|\sum_{j\ge 2s} \sum_{0\le \sigma\le s}2^{-j\la(p_z)}T_j ^k 
B^{k,z}_{j-k-s,\sigma}(\cdot,t)\Big|^2  dt \Big)^{1/2}\Big\|_2^2
\\&\lc
\sum_k\int_I \sum_{j\ge 2s} 2^j \int 
|\phi_j(\rho(2^{-k}t^{-1}\xi)|^2 \Big|\sum_{0\le\sigma\le s} \widehat {B^{k,z}_{j-k-s,\sigma}}(\xi)\Big|^2d\xi\,dt 
\\
&\lc \sum_{k\in \bbZ}\sum_{j\ge 2s} \Big\|\sum_{0\le\sigma\le s}
 B^{k,z}_{j-k-s,\sigma}\Big\|_2^2
 = \sum_k\sum_{m\ge -k+s} 
 \Big\|\sum_{0\le\sigma\le s}
 B^{k,z}_{m,\sigma}\Big\|_2^2
  \lc \|\fS f\|_p^p,
\end{align*}
by Lemma \ref{L2estimatesBsigma}. Similarly, for fixed $j$
\[
\Big\|\Big(\sum_k \int_I  \Big| \sum_{0\le\sigma\le s} 2^{-j\la(p_z)}T_j ^k 
B^{k, z}_{j-k-s,\sigma}(\cdot,t)\big|^2  dt \Big)^{1/2}\Big\|_2^2
\lc  
\sum_k \Big\|\sum_{0\le\sigma\le s}B^{k,z}_{j-k-s,\sigma}\Big\|_2^2\lc\|\fS f\|_p^p.\]
This concludes the proof of \eqref{L2away-sv1-z=1} and  \eqref{L2away-sv2-z=1}.

We now come to the main part of the proof, namely the inequalities 
\eqref{Lqprime-away-1-z=0},  \eqref{Lqprime-away-2-z=0} 
when  $1\le p<2$ and $\Re(z)=0$.
We fix $z$ with $\Re(z)=0$ and then use another interpolation inequality based on 
$$
\text{$[L^1(\ell^1(L^\infty)), L^2(\ell^2(L^2))]_\theta=L^r(\ell^r(L^{r'}))$ for $\theta= 2-\frac 2r$}, $$
where  Calder\'on's complex interpolation method  is applied to  vector-valued $L^p$ spaces 
(see \cite[Theorem 5.1.2]{bergh-lofstrom}.  
As a consequence we  
have
$$\| \cdot\|_{L^r(\ell^r(L^{r'}))} \lc 
\| \cdot \|_{L^1(\ell^1(L^{\infty}))}^{\frac 2r-1} 
\| \cdot \|_{L^2(\ell^2(L^{2}))}^{2-\frac 2r} .
$$
Assuming 
$1\le p<2$, \eqref{Lqprime-away-1-z=0},  \eqref{Lqprime-away-2-z=0} follow from
\begin{subequations}
\begin{multline}\label{L2-away-1-z=0}
\Big\|\Big(\sum_k \int_I \Big|\sum_{j\ge 2s}
\sum_{0\le\sigma\le s}
 2^{-j\la(p_z)}T_j ^k 
B^{k,z}_{j-k-s,\sigma}(\cdot,t)\Big|^{2}  dt \Big)^{1/2}\Big\|_{2}\\ \lc 
2^{-s\frac{d-1}{2}}\alpha^{p/2} \|\fS f\|_p^{p/2},\, \,\,\Re(z)=0,
\end{multline}
\begin{multline}\label{L2-away-2-z=0}
\Big\|\Big(\sum_k \int_I \Big| \sum_{0\le\sigma\le s}2^{-j\la(p_z)}T_j ^k 
B^{k, z}_{j-k-s,\sigma}(\cdot,t)\big|^{2}  dt \Big)^{1/2}\Big\|_{2}\\ \lc 
2^{-j\frac{d-1}{4}}\alpha^{p/2} \|\fS f\|_p^{p/2}, \quad \frac j2\le s\le j,\,\, \Re(z)=0,
\end{multline}
\end{subequations}
and 
\begin{subequations}
\Be\label{L1-away-1-z=0}
\Big\|\sum_k \sup_{t\in I} \Big|\sum_{j\ge 2s} \sum_{0\le\sigma\le s}2^{-j\la(p_z)}T_j ^k 
B^{k,z}_{j-k-s,\sigma}(\cdot,t)\Big|  \Big\|_{1}\lc (1+s)
\|\fS f\|_p^{p},\, \,\, \Re(z)=0,
\Ee
\Be\label{L1-away-2-z=0}
\Big\|\sum_k \sup_{t\in I} \Big| \sum_{0\le\sigma\le s} 2^{-j\la(p_z)}T_j ^k 
B^{k, z}_{j-k-s,\sigma}(\cdot,t)\Big|  \Big\|_{1}\lc  (1+j)
\|\fS f\|_p^{p},\quad \frac j2\le s\le j,\,\, \Re(z)=0.
\Ee
\end{subequations}

This proof of \eqref{L2-away-1-z=0}, \eqref{L2-away-2-z=0} 
is inspired by the work of Christ and Sogge \cite{christ-sogge}, \cite{christ-sogge-survey}.
We use the decomposition \eqref{decompositions} and orthogonality, first in the  $j$-sum  and then, for each $j$, also in the $\nu$ sums, where $\nu\in \cZ_j$. We then see that
\begin{align}
&\Big\|\Big(\sum_k \int_I \Big|\sum_{j\ge 2s}
\sum_{0\le\sigma\le s}
 2^{-j\la(p_z)}T_j ^k 
B^{k,z}_{j-k-s,\sigma}(\cdot,t)
\Big|^{2}  dt \Big)^{1/2}\Big\|_{2}^2
\notag
\\
&\lc 
\sum_k\sum_{j\ge 2s} \sum_{\nu\in \cZ_j} 2^{-j (d-1)}\int_I\Big\|\sum_{0\le \sigma\le s}T_{j,\nu} ^k 
B^{k,z}_{j-k-s,\sigma}(\cdot,t)\Big\|_2^2 \, dt
\notag
\\&=
\sum_m\sum_{k\ge 2s-m}
\sum_{\nu\in \cZ_{k+m}} 2^{-(k+m) (d-1)}\int_I\Big\|\sum_{0\le \sigma\le s}T_{k+m,\nu} ^k 
B^{k,z}_{m-s,\sigma}(\cdot,t)\Big\|_2^2 \, dt.
\label{afterorthogonality}
\end{align}

We use $$\int_I\|T^k_{j,\nu} g\|_2^2dt= \int_I\iint (2^k t)^dh_{j,\nu}(2^kt(x-y))g(y) \overline{g(x)} dy\, dx\, dt,$$
where  $h_{j,\nu}(x)= \cF^{-1}[|\chi_{j,\nu}\varphi_j(\rho(\cdot))|^2](x)$. The kernel
$h_{j,\nu}$  satisfies kernel estimates which are analogous to the right hand side of 
\eqref{Kjnukernelestimate}, i.e.
\[\sup_{t\in I} |t^dh_{j,\nu}(tx)|\lc_N  \frac{2^{-j\frac{d+1}{2}}}
{(1+2^{-j} |x|)^{N} (1+2^{-j/2}|P_{j,\nu}(x)|)^N}.
\]

Using $j=k+m$ we can then estimate, for $t\in I$
\begin{align*} &2^{-(k+m)(d-1)}  \Big\|\sum_{0\le \sigma\le s}T_{k+m,\nu} ^k 
B^{k,z}_{m-s,\sigma}(\cdot,t)\Big\|_2^2 
\,\le C_N \times \\&
\iint 2^{-(k+m)\frac{d-1}{2}} \frac{2^{-md}}{(1+2^{-m}|x-y|)^N} \frac{1}{(1+ 2^{-m+ \frac{k+m}{2}}
| P_{k+m,\nu}(x-y)|)^N}  
|\beta^{k,z}_{m,s}(y)|\,dy\,| \beta^{k,z}_{m,s}(x)|
\, dx,
\end{align*}
with \Be\label{betasdef} \beta^{k,z}_{m,s}:= \sum_{0\le\sigma\le s} B^{k,z}_{m-s,\sigma}.\Ee

Consider a maximal set $Z^s$ of $c2^{-s}$ separated unit vectors $\eta_\varsigma$, and let
$P^s_\varsigma$ be the orthogonal projection to the orthogonal complement of $\nabla\rho(\eta_\varsigma)$.
Notice that for each $\varsigma$ there are $\approx 2^{(d-1) (\frac j2-s)}$ of the vectors $\xi_\nu$ with $\nu\in \cZ_j$ which are of distance $\le C 2^{-s}$ to $\eta_{\varsigma}$.
For those $\nu$ we then have 
$|\frac{\nabla \rho(\xi_\nu)}{|\nabla\rho(\xi_\nu)|}-
\frac{\nabla \rho(\eta_\varsigma)}{|\nabla\rho(\eta_\varsigma)|}|= O(2^{-s})$. Consequently, for those $\nu$, and $j=k+m\ge 2s$
\begin{multline*}
\frac{2^{-md}}{(1+2^{-m}|x-y|)^N} \frac{1}{(1+ 2^{-m+ \frac{k+m}{2}}
| P_{k+m,\nu}(x-y)|)^N}  \\
\lc_N 
\frac{2^{-md}}{(1+2^{-m}|x-y|)^N} \frac{1}{(1+ 2^{-m+ s}
| P^s_{\varsigma}(x-y)|)^N}  
\end{multline*}
and there are $O(2^{(d-1)(\frac{k+m}{2}-s)})$ indices $\nu\in \cZ_{k+m}$ for which we may use this inequality.
Then, setting 
\Be\label{Akmsigma}
A_{k,m,\varsigma}(x) = 
\int 2^{-s(d-1)} \frac{2^{-md}}{(1+2^{-m}|x-y|)^N} \frac{1}{(1+ 2^{-m+ s}
| P^s_{\varsigma}(x-y)|)^N}  
|\beta^{k,z}_{m,s}(y)|\,dy\,, 
\Ee
we get by  the above considerations 
\begin{align}\notag
\eqref{afterorthogonality}& \lc  \sum_{\varsigma\in Z^s}\sum_m\sum_{k\ge 2s-m} 
\int A_{k,m,\varsigma}(x)  |\beta^{k,z}_{m,s}(x)| dx
\\ \label{T*Targ} 
&
\lc \sum_{\varsigma\in Z^s}\sum_m
\int\Big( \sum_{k\ge 2s-m} [A_{k,m,\varsigma}(x) ]^2\Big)^{1/2} \Big(\sum_k |\beta^{k,z}_{m,s}(x)|^2\Big)^{1/2}\,dx.
\end{align}

We first establish  that
\Be\label{Akmsigclaim}
\sup_{m}\sum_{\varsigma\in Z^s}
\Big\|\Big(\sum_{k\ge 2s-m}|A_{k,m,\varsigma}|^2\Big)^{1/2}
\Big\|_\infty \lc  \alpha^p 
2^{-s(d-1)}.
\Ee
For each dyadic cube $Q$  let $y_Q$ be the center of $Q$.
Using \eqref{Akmsigma} we  
estimate for  fixed $x\in \bbR^d$
\begin{multline*}
\Big(\sum_{k\ge 2s-m}|A_{k,m,\varsigma}(x)|^2\Big)^{1/2}\,
\lc \,
2^{-s(d-1)} \times \\\sum_{\substack {Q:\\ L(Q)=m-s}}
\frac{2^{-md}}{(1+2^{-m}|x-y_Q|)^N} \frac{1}{(1+ 2^{-m+ s}
| P^s_{\varsigma}(x-y_Q)|)^N}  \int \Big(\sum_k \Big|\sum_{0\le \sigma\le s} b^{k,z}_{Q,\sigma}(y)
\Big|^2\Big)^{1/2} dy
\end{multline*}
and using
 Lemma \ref{Qaverage} we bound this expression by 
 \begin{align*} & 2^{-s(d-1)} \sum_{\substack {Q:\\ L(Q)=m-s}}
\frac{2^{-md}}{(1+2^{-m}|x-y_Q|)^N} \frac{1}{(1+ 2^{-m+ s}
| P^s_{\varsigma}(x-y_Q)|)^N} \alpha^p|Q| 
\\
&\lc \alpha^p 2^{-s(d-1)} \int \frac{2^{-md}}{(1+2^{-m}|x-w|)^N} \frac{1}{(1+ 2^{-m+ s}
| P^s_{\varsigma}(x-w)|)^N} dw
\lc \alpha^p 2^{-2s(d-1)}.
\end{align*}
We sum over $\varsigma\in Z^s$ and use that $\#Z^s=O(2^{s(d-1)})$  to obtain  \eqref{Akmsigclaim}.

Combining  \eqref{Akmsigclaim} 
and \eqref{T*Targ}  we obtain 
\[\eqref{afterorthogonality}\lc  2^{-s(d-1)} \alpha^p \sum_m
\sum_{\substack {Q\in \fQ_\alpha:\\ L(Q)=m-s}}
\Big\|\Big(\sum_k\Big|\sum_{0\le \sigma\le s} b^{k,z}_{Q,\sigma}\Big|^2\Big)^{1/2}\Big\|_1
.\]
Finally, by Lemma \ref{Qaverage} again
\begin{align*}
 \sum_m
\sum_{\substack {Q\in \fQ_\alpha:\\ L(Q)=m-s}}
\Big\|\Big(\sum_k\Big|\sum_{0\le \sigma\le s} b^{k,z}_{Q,\sigma}\Big|^2\Big)^{1/2}\Big\|_1
\lc \sum_{Q\in \fQ_\alpha} |Q|\alpha^p
\lc \alpha^p|\widetilde \cO_\alpha| \lc \|\fS f\|_p^p,
\end{align*}
by \eqref{measOa}. This finishes the proof of 
\eqref{L2-away-1-z=0}.

The proof of 
 \eqref{L2-away-2-z=0} 
 uses the same idea.
We estimate for fixed $j\in [s/2,s]$, $\Re(z)=0$,
\begin{align}\label{squareofL2-away-2-z=0} 
&\Big\|\Big(\sum_k \int_I \Big| \sum_{0\le\sigma\le s}2^{-j\la(p_z)}T_j ^k 
B^{k, z}_{j-k-s,\sigma}(\cdot,t)\big|^{2}  dt \Big)^{1/2}\Big\|_{2}^2\\ 
&\lc 2^{-j(d-1)} \sum_k \int_I \big\| T^k_j\beta^{k,z}_{j-k,s}(\cdot,t)\big\|_2^2
\lc 2^{-j(d-1)} \sum_{\nu\in \cZ_j} \sum_k \int_I \big\| T^k_{j,\nu}\beta^{k,z}_{j-k,s}(\cdot,t)\big\|_2^2\, dt \notag
\\
&\lc 2^{-j(d-1)} \sum_{\nu\in \cZ_j} \sum_k \int \cA_{k,j,\nu}(x)
\big|\beta^{k,z}_{j-k,s}(x)\big|\, dx\notag,
\end{align}
where again $\beta^{k,z}_{m,s}$ is as in \eqref{betasdef} and
\[\cA_{k,j,\nu}(x):= \int
\frac{2^{kd} 2^{-j\frac{d+1}{2}}}
{(1+2^{k-j}|x-y|)^N (1+ 2^{k-\frac j2}|P_{j,\nu}(x-y)|)^N}|\beta^{k,z}_{j-k,s}(y)|\, dy.
\]
Now $\cA_{k,j,\nu}(x)\lc $
\[
\int
\frac{2^{kd} 2^{-j\frac{d+1}{2}}}
{(1+2^{k-j}|x-w|)^N (1+ 2^{k-\frac j2}|P_{j,\nu}(x-w)|)^N} dw 
\sup_{\substack{Q\in \fQ_\alpha\\L(Q)=j-k-s} }\frac{1}{|Q|}\int\Big|
\sum_{0\le \sigma\le s}
b^{k,z}_{Q,\sigma} (y)\Big| dy
\]
which is bounded by $C\alpha^p$. Consequently
\begin{align*}
&\eqref{squareofL2-away-2-z=0} \lc 2^{-j(d-1)}\sum_{\nu\in \cZ_j} \alpha^p 
\sum_k \sum_{\substack {Q\in \fQ_\alpha\\ L(Q)=j-k-s}} 
\Big\|\sum_{0\le\sigma\le s} b^{k,z}_{Q,\sigma} \Big\|_1
\\
&\lc 2^{-j\frac{d-1}2} \alpha^p\sum_{Q\in \fQ_\alpha}  \Big\|\sum_{0\le \sigma\le s}
b^{j-s-L(Q),z}_{Q,\sigma}
\Big\|_1
\\
&\lc 2^{-j\frac{d-1}2} \alpha^p\sum_{Q\in \fQ_\alpha}  \Big\|\Big(\sum_k\Big|\sum_{0\le \sigma\le s}
b^{k,z}_{Q,\sigma}
\Big|^2\Big)^{1/2}\Big\|_1
\lc 2^{-j\frac{d-1}2} \alpha^p\|\fS f\|_p^p,
\end{align*}
by Lemma \ref{Qaverage}.

We now turn to the proof of \eqref{L1-away-1-z=0},
 \eqref{L1-away-2-z=0}, here  still $\Re(z)=0$.
We estimate the left hand side of 
\eqref{L1-away-1-z=0} using Lemma \ref{Tnklemma} by
\[\sum_k\sum_{j\ge 2s} 2^{-j\frac{d-1}{2}}\sum_{0\le\sigma\le s}
\big\|\sup_{t\in I} \big|
T_j ^k 
B^{k,z}_{j-k-s,\sigma}(\cdot,t)\big|  \big\|_{1} \lc 
\sum_k\sum_{j\ge 2s} \sum_{0\le\sigma\le s}
\big\|
B^{k,z}_{j-k-s,\sigma} \big\|_{1} 
\] and the right hand side of the last display is dominated by 
\begin{align*}
& \sum_{0\le\sigma\le s}\sum_k \sum_{j\ge 2s} \sum_\mu \sum_{\substack{Q\in \fQ_\alpha :\\L(Q)=j-k-s}}
 \sum_{\substack{W\in\Wmq\\ L(W)=-k+\sigma}} \ga_{W,\mu}^{p-1} \Big\|\sum_{\substack{R\in \Rkmw}}  e_R\Big\|_1
 \\&\lc 
 \sum_{0\le\sigma\le s}\sum_k \sum_{j\ge 2s} \sum_\mu 
 \sum_{\substack{Q\in \fQ_\alpha :\\L(Q)=j-k-s}}
  \sum_{\substack{W\in\Wmq\\ L(W)=-k+\sigma}} 
   \ga_{W,\mu}^{p-1}|W|^{1/2}\Big(\sum_{\substack{R\in \Rkmw}} 
 \|e_R\|_2^2\Big)^{1/2}  
\\
   &\lc \sum_{0\le\sigma\le s} \sum_{\mu}\sum_{W\in \fW^\mu}
 \ga_{W,\mu}^p |W|\lc (1+s) \|\fS f\|_p^p. \
\end{align*}

The left hand side of 
\eqref{L1-away-2-z=0}
 is estimated for fixed $j\in [s, 2s]$ by 
\begin{align*} 2^{-j\frac{d-1}{2}}\sum_k \sum_{0\le\sigma\le s} \big\|\sup_{t\in I}
 \Big | T^k_j 
B^{k, z}_{j-k-s,\sigma}
(\cdot,t)\big|  \Big\|_{1}\lc \sum_{0\le \sigma\le s}  \sum_k \|B^{k, z}_{j-k-s,\sigma}\|_1
\end{align*}
and the subsequent estimation is as for \eqref{L1-away-1-z=0}. This concludes the proof of \eqref{Lraway}. \qed

\subsection{\it Proof of \eqref{Lpaway}}\label{Lpawaysect}
It suffices to show, assuming $1\le p<2$, $q=p'$  that for some $a(p,q)>0$ and $s\ge 0$ 
\[ \Big\|\Big(\sum_k \Big[ \int_I \Big|\sum_{j=1}^\infty 2^{-j\la(p)}
\sum_{0\le\sigma\le s}
T_j ^k B^k_{j-k+s,\sigma}(\cdot,t)\Big|^q  dt \Big]^{p/q}\Big)^{1/p}\Big\|_{L^p(\bbR^d\setminus \widetilde \cO_\alpha)}
\lc 2^{-a(p,q)s}
\|\fS f\|_p. \]
When $p>1$ we use the analytic family of functions in \eqref{Bm-complex}.
It suffices to prove the inequalities
\Be\label{L2awayexcept}
\Big\|\Big(\sum_k \int_I  \Big|\sum_{j=1}^\infty 2^{-j\la(p_z)}
\sum_{0\le \sigma\le s}
T_j ^k B^{k,z}_{j-k+s,\sigma}(\cdot,t)\Big|^2 dt\Big)^{1/2} \Big\|_{L^2(\bbR^d\setminus \widetilde \cO_\alpha)}
\lc 
\|\fS f\|_p^{p/2},\quad \Re(z)=1,
\Ee
and
\Be\label{L1awayexcept}
\Big\|\sum_k \sup_{t\in I} \Big|\sum_{j=1}^\infty 2^{-j\la(p_z)}
\sum_{0\le \sigma\le s}
T_j ^k B^{k,z}_{j-k+s,\sigma}(\cdot,t)\Big|\Big\|_{L^1(\bbR^d\setminus \widetilde \cO_\alpha)}
\lc 2^{- \eps s} 
\|\fS f\|_p^p,\quad \Re(z)=0,
\Ee
for some $\eps>0$.

To show \eqref{L2awayexcept} we replace the $L^2(\bbR^d\setminus \widetilde \cO_\alpha)$ norm by the $L^2(\bbR^d)$ norm and argue exactly as in the proof of 
\eqref{L2away-sv1-z=1}, using Lemma \ref{L2estimatesBsigma}.

To show  \eqref{L1awayexcept} it suffices to prove, after Minkowski's inequality for the $\sigma$-summation (involving $O(1+s)$ terms),
\Be\label{L1awayexceptfixedkappa}
\Big\|\sum_k \sup_{t\in I} \Big|\sum_{j=1}^\infty 2^{-j\la(p_z)}
T_j ^k B^{k,z}_{j-k+s,\sigma}(\cdot,t)
\Big|\Big\|_{L^1(\bbR^d\setminus \widetilde \cO_\alpha)}
\lc 2^{- \eps s} 
\|\fS f\|_p^p,\quad \Re(z)=0,\,\,\,0\le \sigma\le s.
\Ee
For the proof observe that, for  $t\in I$,  $T_j^{n,k} B^{k,z}_{j-k+s,\sigma}(\cdot,t)$  is supported in $\widetilde \cO_\alpha$ when $n\le s$ and thus does not contribute to the $L^1(\bbR^d\setminus\widetilde \cO_\alpha)$ norm.
We then use the simple $L^1$ estimate in Lemma \ref{Tnklemma}, part (i), for $n>s$ and 
$\Re(\la(p_z)) =(d-1)/2$ to  estimate the left hand side of \eqref{L1awayexceptfixedkappa} by a constant times
\begin{align*}
&2^{-sN} \sum_k\sum_j  \big\|B^{k,z}_{j-k+s,\sigma}\big\|_{1}
\\&\lc2^{-sN}
\sum_k\sum_j\sum_{\substack {Q\in \fQ_\alpha\\ L(Q)=j-k+s}}
\sum_\mu\sum_{\substack{W\in \Wmq\\ L(W)=-k+\sigma}} \ga_{W,\mu}^{p-1}
\Big\|\sum_{\substack{R\in \Rkmw}}e_R\Big\|_1.
\end{align*}
We interchange the sums and note that each $W$ is contained in a unique cube $Q\in \fQ_\alpha$, and thus because of the disjointness of the cubes in $\fQ_\alpha$
 the $(j,Q)$ sums corresponding to a fixed  $W$  
 collaps to a single term. Hence we can bound the  previous expression by $C_N$ times
\begin{align*}
&2^{-sN} \sum_k
\sum_\mu\sum_{\substack{W\in \fW^\mu\\L(W)=-k+\sigma}} \ga_{W,\mu}^{p-1}
|W|^{1/2} \Big(\sum_{\substack{R\in \Rkmw}}\|e_R\|_2^2\Big)^{1/2}
\\ &\lc 2^{-sN} 
\sum_\mu\sum_{W\in \fW^\mu} \ga_{W,\mu}^{p}|W| \lc 2^{-sN}\|\fS f\|_p^p.
\end{align*}
This completes the proof of \eqref{Lpaway}. \qed

\subsection{\it Proof of \eqref{Lpatomic}}\label{Lpatomicsect}
The estimate  follows from the inequalities
\Be\label{LpatomicOmmu}
\Big\|\Big(\sum_k \Big[\int_I \Big|\sum_{j\ge 1} 2^{-j\la(p)}
\sum_{\substack{m, \sigma:\\ \sigma>| m+k-j|, \\ \sigma\ge j }}
T_j ^{k} B^k_{m,\sigma}(\cdot,t)\Big|^q  dt \Big]^{2/q}
\Big)^{1/2}
\Big\|_{L^p(\bbR^d\setminus \widetilde \cO_\alpha)}
\lc  \|\fS f\|_p  
\Ee
and
\Be\label{Lpatomickappaoff}
\Big\|\Big(\sum_k \Big[\int_I \Big|\sum_{j\ge 1}^\infty 2^{-j\la(p)}
\sum_{\substack{m,\sigma:\\ \sigma>| m+k-j| \\ \sigma<j}}
T_j ^{k} B^k_{m,\sigma}(\cdot,t)\Big|^q  dt \Big]^{p/q}
\Big)^{1/p}
\Big\|_{L^p(\bbR^d)}\\
\lc  \|\fS f\|_p.
\Ee

\subsubsection{Proof of \eqref{LpatomicOmmu}}
We use the decomposition $T^k_j=\sum_{n>0} T^{n,k}_j$ and use Minkowski's inequality for the $j$ and $n$ sums.  
When $j+n\le \sigma$ the support of
$T_j ^{k,n} B^k_{m,\sigma}(\cdot,t)$ 
is contained in $\widetilde \cO_\alpha$, for all $t\in I$. Thus in 
\eqref{LpatomicOmmu}  we only need to consider the terms  with $|m+k-j|<\sigma$ and $
j\le \sigma\le j+n$. 
Since $\la(p)+1/q>0$ it suffices to show for fixed $j\ge 1$, that
\Be\label{LpatomicOmmu-fixedjn}
\Big\|\Big(\sum_k \Big[\int_I \Big|
\sum_{\substack{m, \sigma:\\ \sigma>| m+k-j|\\ j\le \sigma\le j+n}}
T_j ^{n,k} B^k_{m,\sigma}(\cdot,t)
\Big|^q  dt \Big]^{2/q}
\Big)^{1/2}
\Big\|_{L^p(\bbR^d)} \lc 2^{-n} 2^{-j/q}
  \|\fS f\|_p  .
\Ee

This follows from
\Be\label{LpatomicOmmu-fixedjn-z-1}
\Big\|\Big(\sum_k \Big[\int_I \Big|
\sum_{\substack{m, \sigma:\\ \sigma>| m+k-j|\\ j\le \sigma\le j+n}}
T_j ^{n,k} B^{k,z}_{m,\sigma}(\cdot,t)\Big|^q  dt \Big]^{2/q}
\Big)^{1/2}
\Big\|_{L^2(\bbR^d)} \lc 2^{-n} 2^{-j/q}
\ \|\fS f\|_p^{p/2} ,\,\,\, \Re(z)=1,
\Ee
and
\Be\label{LpatomicOmmu-fixedjn-z-0}
\Big\|\Big(\sum_k \Big[\int_I \Big|
\sum_{\substack{m, \sigma:\\ \sigma>| m+k-j|\\ j\le \sigma\le j+n}}
T_j ^{n,k} B^{k,z}_{m,\sigma}(\cdot,t)\Big|^q  dt \Big]^{2/q}
\Big)^{1/2}
\Big\|_{L^1(\bbR^d)} \lc 2^{-n} 2^{-j/q}
  \|\fS f\|_p^p , \,\,\,\Re(z)=0.
\Ee
By Lemma \ref{Tnklemma}, part (iii), 
the left hand side of 
\eqref{LpatomicOmmu-fixedjn-z-1}
  is 
\begin{align*}
&\Big(\sum_k \Big\|\Big(\int_I \Big|
\sum_{\substack{m, \sigma:\\ \sigma>| m+k-j|\\ j\le \sigma\le j+n}}
T_j ^{n,k} B^{k,z}_{m,\sigma}(\cdot,t)\Big|^q  dt \Big)^{1/q}\Big\|_2^2
\Big)^{1/2}
\lc 2^{-n-j/q}
\Big(\sum_k \Big\|
\sum_{\substack{m, \sigma:\\ \sigma>| m+k-j|\\ j\le \sigma\le j+n}}
 B^{k,z}_{m,\sigma}\Big\|_2^2\Big)^{1/2}.
 \end{align*}
 Recall that \[\text{supp}(B^{k,z}_{m,\sigma}) \subset \bigcup_{\substack{Q\in \fQ_\alpha\\L(Q)=m}}Q.\] 
 Therefore,  for   $\Re(z)=1$ we have
 \begin{align*}
 &\Big(\sum_k \Big\|
\sum_{\substack{m, \sigma:\\ \sigma>| m+k-j|\\ j\le \sigma\le j+n}}
 B^{k,z}_{m,\sigma}\Big\|_2^2\Big)^{1/2}
 =
 \Big(\sum_k \sum_{m\ge -k}\Big\|
\sum_{\substack{ \sigma:\\ \sigma>| m+k-j|\\ j\le \sigma\le j+n}}
 B^{k,z}_{m,\sigma}\Big\|_2^2\Big)^{1/2}\lc \|\fS f\|_p^p,
 \end{align*}
 by Lemma \ref{L2estimatesBsigma}. Hence \eqref{LpatomicOmmu-fixedjn-z-1} follows.


 We now turn to the proof of 
 \eqref{LpatomicOmmu-fixedjn-z-0}, where $\Re(z)=0$.
For $W\in \Wbad$ let   $Q(W)$ be  the unique cube in $\fQ_\alpha$ containing $W$.
We can split $$B^{k,z}_{m,\sigma} =\sum_{\mu\in \bbZ}\sum_{\substack{W\in \fW^\mu_{\text{bad}}\\L(W)=-k+\sigma}}  \widetilde B^{k,z}_{m,\sigma, \mu,W},$$ where 
 \[\widetilde B^{k,z}_{m,\sigma, \mu, W}  = 
 \begin{cases} 
 \gamma_{W,\mu,z}
\sum_{\substack{R\in \Rkmw}}e_R, \,\,&\text{ if } L(Q(W))=m \text{ and } L(W)=-k+\sigma,
\\
0 &\text{ if  either $L(Q(W))\neq m$ or $L(W)\neq -k+\sigma$}. 
\end{cases}
\]
 	Observe that for $j\le \sigma$, $L(W)=-k+\sigma$,  the function
	$T^{n,k}_j 
	\widetilde B^{k,z}_{m,\sigma, \mu,W}$ is supported in a  $2^{n+3}$-dilate of $W$ (with respect to its center).
	Hence, by the Minkowski and  Cauchy-Schwarz inequalities 
	we estimate for fixed $j,n$ 
	\begin{align*}
&\Big\|\Big(\sum_k \Big[\int_I \Big|
\sum_{\substack{m, \sigma:\\ \sigma>| m+k-j|\\ \sigma\le j+n}}
T_j ^{n,k}  B^{k,z}_{m,\sigma}(\cdot,t)\Big|^q  dt \Big]^{2/q}
\Big)^{1/2}
\Big\|_{L^1(\bbR^d)} 
\\
&\lc \sum_{\mu}\sum_{W\in \Wbad }
2^{nd/2} | W|^{1/2}
\Big\|\Big(\sum_{\substack{k : |L(Q(W))+k-j|\\ <L(W)+k
\le j+n} }\Big[\int_I \Big|
T_j ^{n,k} \widetilde B^{k,z}_{L(Q(W)),L(W)+k,\mu, W}(\cdot,t)\Big|^q  dt \Big]^{2/q}
\Big)^{1/2}
\Big\|_2 
\end{align*} which by an application of Lemma \ref{Tnklemma}
is bounded by 
\begin{align*}
&C_N  2^{-n(N-d/2)} 2^{-j/q}\sum_{\mu}\sum_{W\in \Wbad }
| W|^{1/2}
\Big(\sum_{\substack{k : \\|L(Q(W))+k-j|<L(W)+k\\
L(W)+k\le j+n} }\Big\|
\gamma_{W,\mu,z}
\sum_{\substack{R\in \Rkmw}}e_R\Big\|_2^2\Big)^{1/2}
\\
&\lc 2^{-n-j/q} 
   \sum_{\mu}\sum_{W\in \fW^{\mu} }
   | W|^{1/2}\ga_{W,\mu}^{p-1}
\Big(\sum_k 
\sum_{\substack{R\in \Rkmw}}\|e_R\|_2^2
 \Big)^{1/2}
 \\
&\lc 2^{-n-j/q} \sum_{\mu}\sum_{W\in \fW^{\mu} } |W|\ga_{W,\mu}^p 
\lc 2^{-n-j/q}\|\fS f\|_p^p.
  \end{align*}

 \subsubsection{Proof of \eqref{Lpatomickappaoff}}
 By Minkowski's inequality \eqref{Lpatomickappaoff} follows if we can prove for fixed $\sigma>0$,
\begin{multline}\label{Lpatomickappa}
\Big\|\Big(\sum_k \Big[\int_I \Big|\sum_{j>\sigma} 2^{-j\la(p)}
\sum_{\substack{m:\\ \sigma>| m+k-j|}}T_j ^k B^k_{m,\sigma}(\cdot,t)\Big|^q  dt \Big]^{p/q}
\Big)^{1/p}
\Big\|_{L^p(\bbR^d)}\\
\lc (1+\sigma)^{1/p} 2^{-\sigma d(\frac 1q-\frac 1{p'})} \|\fS f\|_p.
\end{multline}
When $p>1$ we use  complex interpolation to deduce this from
\begin{multline}\label{Lpatomickappa-z=1}
\Big\|\Big(\sum_k \int_I \Big|\sum_{j>\sigma} 2^{-j\la(p_z)}
\sum_{\substack{m:\\ \sigma>| m+k-j|}}T_j ^k B^{k,z}_{m,\sigma}(\cdot,t)\Big|^2  dt 
\Big)^{1/2}
\Big\|_{L^2(\bbR^d)}\\
\lc (1+\sigma)^{1/2} \|\fS f\|_p^{p/2} ,\quad\Re(z)=1,
\end{multline}
and,  with $\frac{1}{q_0}= (\frac1p-\frac 1{q'})/(\frac 2p-1)$,
\begin{multline}\label{Lpatomickappa-z=0}
\Big\|\sum_k \Big(\int_I \Big|\sum_{j>\sigma} 2^{-j\la(p_z)}
\sum_{\substack{m:\\ \sigma>| m+k-j|}}T_j ^k B^{k,z}_{m,\sigma}(\cdot,t)\Big|^{q_0}  dt \Big)^{1/q_0}
\Big\|_{L^1(\bbR^d)} \\
\lc (1+\sigma) 2^{-\sigma  d/{q_0}} \|\fS f\|_p^p, \quad\Re(z)=0.
\end{multline}
Note that $1/q_0=1-1/r$ where $r$ is as in \eqref{rdefinition}, and we have
$(1-\vth)(1,\frac{1}{q_0})+\vth (\frac 12,\frac 12)=(\frac 1p,\frac 1q)$ for $\vth=2/p'$.

We first consider the inequality for $\Re(z)=1$.
We can use the orthogonality of the functions $\varphi_j(\rho(\cdot/t))$ to estimate 
\begin{align*}
&\Big\|\Big(\sum_k \int_I \Big|\sum_{j>\sigma} 2^{-j\la(p_z)}
\sum_{\substack{m:\\ \sigma>| m+k-j|}}T_j ^k B^{k,z}_{m,\sigma}(\cdot,t)\Big|^2  dt 
\Big)^{1/2}
\Big\|_{2}
\\
&\lc \Big(\sum_k \int_I \sum_{j>\sigma } 2^{j}\Big\|
\sum_{\substack{m:\\ \sigma>| m+k-j|}}T_j ^k B^{k,z}_{m,\sigma}(\cdot,t)\Big\|_2^2  dt 
\Big)^{1/2}
\lc \Big(\sum_{k}\sum_{j>\sigma}  
\Big\|\sum_{\substack{m:\\ \sigma>| m+k-j|}}
B^{k,z}_{m,\sigma}\Big\|_2^2  
\Big)^{1/2}.
\end{align*} 
We use the disjointness of the cubes in $\fQ_\alpha$ and then interchange the $m,j$ summations. Using  that for fixed $m,k$ there are $O(1+\sigma)$ terms in the $j$ summation, we bound the last expression by
\begin{align*}\Big(\sum_{k}\sum_{j>\sigma}  
\sum_{\substack{m\ge -k:\\ \sigma>| m+k-j|}}\big\| B^{k,z}_{m,\sigma}\big\|_2^2  
\Big)^{1/2} &\lc (1+\sigma)^{1/2}
\Big(\sum_{k}
\sum_{m\ge -k}\big\| B^{k,z}_{m,\sigma}\big\|_2^2  \Big)^{1/2}
\\& \lc   (1+\sigma)^{1/2}\|\fS f\|_p^{p/2},
\end{align*}
where in the last line we have applied Lemma \ref{L2estimatesBsigma} to conclude \eqref{Lpatomickappa-z=1}.


We now turn to \eqref{Lpatomickappa-z=0}. 
We split $T^k_j=\sum_{n=0}^\infty T^{n,k}_j$, 
set $$b_{W,\mu, z}^k =\ga_{W,\mu,z}
\sum_{\substack{R\in \Rkmw}}e_R
$$
 and estimate the left hand side of  \eqref{Lpatomickappa-z=0} by
\Be\label{tagLHS}
\sum_{k} \sum_{n\ge 0}
\sum_{j>\sigma} 2^{-j\frac{d-1}{2}}
\sum_{\substack{m:\\ \sigma>| m+k-j|}} \sum_{\substack{ Q\in \fQ_\alpha\\L(Q)=m}}
\sum_\mu\sum_{\substack{W\in \Wmq\\L(W)=-k+\sigma}}   \Big\|
\Big(\int_I \big|
T_j ^{n,k} b_{W,\mu,z}^k(\cdot,t)
\big|^{q_0}  dt \Big)^{1/q_0}
\Big\|_{1}.
\Ee We bound  for fixed $W$, with $L(W)=-k+\sigma$,
\begin{align*}
&\Big\|
\Big(\int_I \big|
T_j ^{n,k} b_{W,\mu,z}^k(\cdot,t)
\big|^{q_0}  dt \Big)^{1/q_0}
\Big\|_{1}
\\
&\lc 2^{(-k+j+n)d/q_0}  \Big\|
\Big(\int_I \big|
T_j ^{n,k}b_{W,\mu, z}^k
(\cdot,t)
\big|^{q_0}  dt \Big)^{1/q_0}
\Big\|_{q_0'}
\\
&
\lc 2^{(-k+j+n)d/q_0}  2^{j(d(\frac{1}{q_0'}-\frac 12)-\frac 12)}2^{-nN} 
\|b_{W,\mu, z}^k\|_{q_0'},
\end{align*}
by Lemma \ref{Tnklemma}, part (ii).
Hence after summing in $n$
\begin{align*}
(\ref{tagLHS})\lc
 \sum_{k} 
\sum_{j>\sigma} 
\sum_{\substack{m:\\ \sigma>| m+k-j|}}\sum_{\substack{Q\in \fQ_\alpha\\L(Q)=m}}
\sum_\mu\sum_{\substack{W\in \Wmq\\L(W)=-k+\sigma}}
2^{-kd/q_0} 
\|b^k_{W,\mu,z}\|_{q_0'}.
\end{align*}
Observe that  for $L(W)+k=\sigma$,
\begin{align*}
&2^{-kd/q_0} 
\|b^k_{W,\mu,z}\|_{q_0'}\le
2^{-kd/q_0} |W|^{1/q_0'-1/2}
\|b^k_{W,\mu,z}\|_{2}
\\&\lc 2^{-\sigma d/q_0}|W|^{1/2} \gamma_{W,\mu}^{p-1}
\Big(\sum_{R\in \Rkmw} \|e_R\|_2^2\Big)^{1/2} 
\le 2^{-\sigma d/q_0}|W| \gamma_{W,\mu}^{p}.
\end{align*} 
We interchange summations  and use that, for fixed $W\in \fW^\mu_{\text{bad}}$, 
\[\#\{j\ge \sigma : | L(Q(W))+\sigma-L(W)-j|<\sigma\} =O(1+\sigma).\]
We then   obtain
\begin{align*}
(\ref{tagLHS})&\lc
 2^{-\sigma d/q_0} (1+\sigma)
\sum_\mu\sum_{W\in \fW^\mu}
|W| \gamma_{W,\mu}^{p}
\lc 2^{-\sigma d/q_0} (1+\sigma)\|\fS f\|_p^p.
\end{align*}
This  completes the proof of 
\eqref{Lpatomickappaoff}, and then 
\eqref{Lpatomic}  and  finally the proof of Theorem \ref{mainthm}.

\section{Proofs of Theorems \ref{H1thmTd} and \ref{LpthmTd}}
\label{transferencesect}
In this section we use Theorems \ref{H1thm} and \ref{Lpthm} proved in $\bbR^d$ and transference argument to establish the corresponding versions for periodic functions. Such transference arguments go back to 
De Leeuw \cite{deleeuw}. See also \cite{kenig-tomas} for transference of maximal operators and 
\cite{liu-lu},  \cite{fan-wu}  inequalities in Hardy spaces on $\bbT^d$. In our presentation we rely on the method in \cite{fan-wu}.

\subsection{\it   The $h^1(\bbT^d)\to L^{1,\infty}(\bbT^d) $ bound}
\label{h1transference}
We identify functions $f$ on $\bbT^d$ with functions on $\bbR^d$ satisfying 
$f(x+n)=f(x)$ for all $n\in \bbZ^d$. Let $Q^0=[-\frac 12, \frac 12]^d$. 

Let 
 $$h_\la(s)= (1-\upsilon_0(s)) (1-s)_+^\la$$
and 
$\cS^\la_t f=\sum_{\ell\in \bbZ^d}
h(\rho(\ell/t)) \inn{f}{e_\ell}e_\ell$. Let $\la(1)=\frac{d-1}{2}$. 
After  a reduction analogous to the one in \S\ref{origincontr} we need to prove the bound
$$\Big\| \Big(\sum_{k>0}\int_I |\cS^{\la(1)}_{2^k t} f|^q dt\Big)^{1/q}
\Big\|_{L^{1,\infty}(\bbT^d)} \lc \|f\|_{h^1(\bbT^d)}.$$
By normalizing we may assume that $\|f\|_{h^1(\bbT^d)} =1$.

By the atomic decomposition for periodic functions (\cite{goldberg}, \cite{fan-wu})
 we may assume that
$$ f= f_0 + \sum_{Q\in \cQ} c_Q a_Q,$$
where $f_0\in L^2$, $\|f_0\|_2\lc 1$, 
where $\cQ$ is a collection of cubes of sidelength at most $1/4$  which intersect the fundamental cube $Q^0$ and where  $a_Q$ is periodic and supported in $Q+\bbZ^d$, satisfying 
$\|a_Q\|_{L^2(Q^0)} \le |Q|^{-1/2}$ and 
\Be\label{polcanc}\int_Q a_Q(x) P(x) dx=0\Ee for all polynomials of degree at most $2d$.
Moreover \Be\label{H1normal}\|f\|_{h^1} \approx \|f_0\|_2+\sum_{Q\in \cQ} |c_Q| \approx 1.\Ee
The contribution acting on $f_0$ is taken care of by standard $L^2$ estimates. 

Now let $\gamma=(\gamma_1,\dots, \gamma_d)\in \{-\frac 12, 0,\frac 12\}^d=:\Gamma$ and let $Q^\gamma=\gamma+Q^0$. We can then split the family of cubes $\cQ$
 into $3^d$ disjoint families $\cQ_\gamma$ so that each cube $Q\in \cQ_\ga$ has the property that its double is contained in the cube $Q^\gamma$.
By periodicity, and the monotone convergence theorem,  it suffices to prove for each finite subset $\cN $ of $\bbN$, and for each $\gamma\in \Gamma$,
\Be \label{finitenumberofks}
 \sup_{\alpha>0} \alpha \,\,\meas\Big(\Big\{x\in Q^\ga: 
 \Big(\sum_{k\in \cN}\int_I \big |\cS^{\la(1)}_{2^k t} [\sum_{Q\in \cQ^\gamma} c_Qa_Q]\big|^q dt\Big)^{1/q}>\alpha\Big\}\Big)
\lc 1.
\Ee
It suffices to show for every {finite} subset $\sF^\gamma$ of $\cQ^\gamma$
\Be\label{finitesFgamma}
 \sup_{\alpha>0} \alpha \,\,\meas\Big(\Big\{x\in Q^\ga: 
 \Big(\sum_{k\in \cN}\int_I \big |\cS^{\la(1)}_{2^k t} [\sum_{Q\in \sF^\gamma} c_Qa_Q]\big|^q dt\Big)^{1/q}>\alpha\Big\}\Big)
\lc \sum_{Q\in \sF^\ga}|c_Q|,
\Ee
where the implicit constant is independent of $\sF^\ga$.
To see the reduction we split $\cQ_\gamma= \cup_{n=0}^\infty \sF^{\ga,n}$, where $\sF^{\ga,n}$ is finite and $\sum_{Q\in \sF^{\ga,n}}|c_Q|\le 2^{-n}$. 
By using the result of  Stein and N. Weiss on adding $L^{1,\infty}$ functions \cite[Lemma 2.3]{stein-nweiss}
the left hand side in \eqref{finitenumberofks} is bounded by 
$C\sum_{n=0}^\infty (1+n)2^{-n}\lc 1$, as  claimed.

In order to prove \eqref{finitesFgamma} we can renormalize again,  replacing $c_Q$ with $c_Q (\sum_{Q'\in \sF^\gamma}|c_{Q'}|)^{-1}$ 
and $\alpha$ with $\alpha
(\sum_{Q'\in \sF^\gamma}|c_{Q'}|)^{-1}$. It therefore remains to prove
for every finite subset $\sF^\gamma$ of $\cQ^\gamma$, and for 
$\sum_{Q\in \sF^\gamma}|c_Q|=1$, that
\Be\label{finitesFgammanormal}
 \sup_{\alpha>0} \alpha \,\,\meas\Big(\Big\{x\in Q^\ga: 
 \Big(\sum_{k\in \cN}\int_I \big |\cS^{\la(1)}_{2^k t} [\sum_{Q\in \sF^\gamma} c_Qa_Q]\big|^q dt\Big)^{1/q}>\alpha\Big\}\Big)
\lc 1,
\Ee
where the implicit constant is independent of $\sF^\ga$.

Now fix $\alpha>0$.  Let $\phi\in C^\infty$ supported in $\{x:|x|\le 1\}$ 
such that $\int\phi(x) dx=1$ and let $\phi_\eps=\eps^{-d}\phi(\eps^{-1}\cdot)$.
Choose $\eps_Q$ be small, less than  one tenth of the sidelength of $Q$  so that in addition 
$\|\phi_{\eps(Q)}*a_Q-a_Q\|_2< \alpha^{1/2} $. 
Let $\ta_Q=\phi_{\eps(Q)}*a_Q$.
Then by Tshebyshev's inequality and standard $L^2$ estimates
(such as in \S\ref{prelsect}) 
\begin{align*}&\meas\Big(\Big\{x\in Q^\ga: 
 \Big(\sum_{k\in \cN}\int_I \Big |\cS^{\la(1)}_{2^k t}\big [\sum_{Q\in \sF^\gamma} c_Q(a_Q-\ta_Q)\big]\Big|^q dt\Big)^{1/q}>\alpha\Big\}\Big)
\\
&\lc \alpha^{-2}
\Big\|
 \Big(\sum_{k\in \cN}\int_I \Big |\cS^{\la(1)}_{2^k t} \big[\sum_{Q\in \sF^\gamma} 
 c_Q(a_Q-\ta_Q)\big]\Big|^q dt\Big)^{1/q}\Big\|_2^2
\\
&\lc  \alpha^{-2} \Big(\sum_{Q\in \sF^\gamma} |c_Q| \|a_Q-\ta_Q\|_2\Big)^2
\lc \alpha^{-1}  \Big(\sum_{Q\in \sF^\gamma} |c_Q| \Big)^2 \lc \alpha^{-1};
\end{align*}
here we have used the normalization
$\sum_Q |c_Q|\le 1$.

It suffices to show that 
\Be\label{sufficesone}
\meas\Big(\Big\{x\in Q^\ga: 
 \Big(\sum_{k\in \cN}\int_I \Big |\cS^{\la(1)}_{2^k t}\big [\sum_{Q\in \sF^\gamma} c_Q\ta_Q\big]\Big|^q dt\Big)^{1/q}>\alpha\Big\}\Big)
\lc \alpha^{-1}.
\Ee
We shall now follow the argument in \cite{fan-wu} and set
\Be \label{Psidef}\Psi(x)=\prod_{i=1}^d(1-x_i^2/4)_+, \quad \Psi^\ga_N(x)= \Psi(N^{-1} (x-\ga)).
\Ee
As in   \cite{fan-wu} we use the following  formula, which is valid at least 
for $g$ in the Schwartz space of $\bbT^d$,  for all $x\in \bbR^d$.
\begin{multline}\label{commutformula}
\Psi^\ga_N(x) \cS^\la_{2^k t}g(x)- S^\la_{2^k t}[\Psi^\ga_N g](x)\,=
\\
\sum_{\ell\in \bbZ^d} \inn{g}{e_\ell}e_\ell(x) \int \Big[
h_\la(\rho(\frac{\ell}{2^kt}))
-h_\la(\rho(\frac{\ell+N^{-1}\xi}{2^kt}))
\Big] \widehat \Psi(\xi) 
e^{2\pi i \inn{x-\gamma}{N^{-1}\xi}} d\xi.
\end{multline}
As the Fourier coefficients $\inn{g}{e_\ell}$ decay rapidly,  $\widehat \Psi\in L^1$ and $h_\la$ is H\"older continuous for $\la>0$ this  implies 
\Be\label{limitrelation}
\lim_{N\to \infty} \sup_{t\in I}
 \sup_{x\in \bbR^d}
\big|
\Psi^\gamma_N(x) \cS^{\la(1)}_{2^k t} g(x)- 
S^{\la(1)}_{2^k t} [\Psi^\gamma_N g] (x)\big| =0,
\Ee
for $k\in \cN$.

Next we observe that $\Psi^\ga_N(x) \ge (3/4)^d$ for all $x\in m+Q^\ga$, when
$-N\le m_i\le N$ for $i=1,\dots, d$.
Using periodicity we see that the left hand side of 
\eqref{sufficesone} is equal to
\begin{align*}
&(2N+1)^{-d}\sum_{\substack {-N\le m_i\le N\\i=1,\dots, d}}
\meas\Big(\Big\{x\in m+Q^\ga: 
 \Big(\sum_{k\in \cN}\int_I \Big |\cS^{\la(1)}_{2^k t} \big[\sum_{Q\in \sF^\gamma} c_Q\ta_Q\big ](x)\Big|^q dt\Big)^{1/q}>\alpha\Big\}\Big)
\\
&\le (2N+1)^{-d}
\meas\Big(\Big\{x\in \bbR^d: 
 \Big(\sum_{k\in \cN}\int_I \Big |\Psi^\ga_N(x) \cS^{\la(1)}_{2^k t} \big[\sum_{Q\in \sF^\gamma} c_Q\ta_Q\big](x)\Big|^q dt\Big)^{1/q} 
 >(3/4)^d\alpha\Big\}\Big).
\end{align*}
Consider the periodic $C^\infty$  function
 $g=\sum_{Q\in \sF^\gamma} c_Q\ta_Q$ and apply 
\eqref{limitrelation}. 
Hence there is $N_0=N_0(g,\alpha,\cN)$ such that  for every $x\in \bbR^d$ and $N> N_0$,
$$\Big(\sum_{k\in \cN}\int_I \Big |\Psi^\ga_N(x) 
\cS^{\la(1)}_{2^k t}  \big[\sum_{Q\in \sF^\gamma} c_Q\ta_Q\big ](x)
- S^{\la(1)}_{2^k t}  \big[\Psi^\ga_N \sum_{Q\in \sF^\gamma} c_Q\ta_Q\big](x)
\Big|^q dt\Big)^{1/q}  < (3/4)^d \alpha/2.
$$

Assuming $N>N_0$ in what follows we see that it suffices to bound
\Be\label{Rdclaim}
(2N+1)^{-d}
\meas\Big(\Big\{x\in \bbR^d: 
 \Big(\sum_{k\in \cN}\int_I \Big |S^{\la(1)}_{2^k t}\big [
 \Psi^\gamma_N\sum_{Q\in \sF^\gamma} c_Q\ta_Q
 \big](x)\Big|^q dt\Big)^{1/q} 
 >(3/4)^d\alpha/2\Big\}\Big).
\Ee
Define for $Q\in \sF^\gamma$, $m\in \bbZ^d$,
$$a_{Q,m} (y) = \bbone_{m+Q^\ga}(y) \Psi^\gamma_N(y) \widetilde a_Q(y).$$
Then the support of $a_{Q,m}$ is in the interior of $m+Q^\ga$ and $\Psi^\gamma_N$ coincides on the support of $a_{Q,m}$ with a bounded polynomial of degree $2d$. Hence $a_{Q,m}$ is an $L^2$ function supported on the double of $Q$, such that
$\int a_{Q,m}(y) dy=0$ and such that $\|a_{Q,m}\|_2\lc |Q|^{-1/2}$. Moreover $a_{Q,m}$ is nontrivial only when $|m_i|\le 2N$ for $i=1,\dots, d$. This implies 
$$
\Big\|
\Psi^\gamma_N\sum_{Q\in \sF^\gamma} c_Q\ta_Q \Big\|_{H^1(\bbR^d)} 
\le \sum_{\substack {-2N\le m_i\le 2N\\i=1,\dots, d}} \sum_{Q\in \sF^\gamma} |c_Q| \|a_{Q,m}\|_{H^1(\bbR^d)} 
\lc (4N+1)^d.
$$
We now apply Theorem \ref{H1thm} to see that the left hand side of \eqref{Rdclaim} is bounded by $$
C\alpha^{-1}(2N+1)^{-d} \Big\|
\Psi^\gamma_N\sum_{Q\in \sF^\gamma} c_Q\ta_Q \Big\|_{H^1(\bbR^d)} 
\lc  (2N+1)^{-d} (4N+1)^d \alpha^{-1}\lc \alpha^{-1}$$ which finishes the proof of the theorem.  \qed

\subsection{\it The $L^p\to L^{p,\infty}$ bound}\label{ptransference}
The proof is similar (but more straightforward), therefore we will be brief.
Now $\la(p)$ can be negative, but we have $\la>-1/q$.
The limiting relation \eqref{limitrelation} is now replaced by 
\Be\label{limitrelation2}
\lim_{N\to \infty} 
\sup_{x\in \bbR^d}
\Big(\int_I 
\big|
\Psi^\gamma_N(x) \cS^{\la(p)}_{2^kt} g(x)- 
S^{\la(p)}_{2^kt} [\Psi^\gamma_N g](x)\big|^q dt\Big)^{1/q} =0, \qquad k\in \cN.
\Ee
Here, we consider $g\in \cS(\bbT^d)$. We sketch a proof of \eqref{limitrelation2}, based on \eqref{commutformula}.

We start by observing that 
\Be \label{ob_uni}
 \int_I |h_\la(\rho(\zeta /t))|^q dt \leq C, 
\Ee
uniformly in $\zeta \in \R^d$. To see this, note that $\rho(\zeta /t) = \rho(\zeta) t^{-1/b}$ and we may assume that $\rho(\zeta)  \sim 1$ due to the support of $h_\la$. Therefore, \eqref{ob_uni} follows by a change of variable. From this observation, we may reduce \eqref{limitrelation2} to 
\Be \label{limitrelation2_red}
\lim_{N\to \infty}  \Big(\int_I \big|
h_\la(\rho(\frac{\ell}{2^kt}))
-h_\la(\rho(\frac{\ell+N^{-1}\xi}{2^kt}))\big|^q dt\Big)^{1/q} =0
\Ee
for fixed $l,k,\xi$ using \eqref{commutformula}, Minkowski's inequality and the dominated convergence theorem.

For \eqref{limitrelation2_red}, we argue as follows. Let $h\in L^q(J)$ for a compact subinterval $J$ of $(0,\infty)$. Then 
for any
$a>0$
\[
 \lim_{\delta\to 0}\Big(\int_J | h(as) - h((a+\delta)s) |^q ds \Big)^{1/q} = 0 
\] and the limit is uniform if $a$ is taken from a compact subset of $(0,\infty)$.
This is easily seen for smooth $h$ and follows for general $h\in L^q(J)$ by an approximation argument. Changing variables $s=t^{-1/b}$ we obtain that for any compact subinterval $I\subset (0,\infty)$
\Be\label{dilcont}\lim_{\delta\to 0}  \Big(\int_I | h(at^{-1/b}) - h((a+\delta)t^{-1/b}) |^q dt \Big)^{1/q} =0.\Ee
Then \eqref{limitrelation2_red} follows from \eqref{dilcont} with $h=h_\la$, 
$\delta = \rho((\ell+N^{-1}\xi)/2^k)-\rho(\ell/2^k)$ and  $a= \rho(\ell/2^k)$ using the homogeneity and continuity of $\rho$.

Finally,  using  \eqref{limitrelation2} we get, for sufficiently large $N$,
\begin{align*} 
&\meas\Big(\Big\{x\in Q^0: 
 \Big(\sum_{k\in \cN}\int_I \big |\cS^{\la(p)}_{2^k t} g\big|^q dt\Big)^{1/q}>\alpha\Big\}\Big) \\
 &\lc  (2N+1)^{-d} \meas\Big( \Big\{ x\in \bbR^d: \Big(\sum_{k\in \cN}\int_I  \big|
 S^{\la(p)}_{2^k t} [\Psi^0_N g](x)\big |^q dt\Big)^{1/q} > (3/4)^d \alpha/2\Big\}\Big)
 \end{align*}
and by  Theorem \ref{Lpthm} we bound the right hand side by 
\[ C(2N+1)^{-d} \alpha^{-p} \|\Psi_N^0 g\|^p_{L^p(\bbR^d)}\lc \alpha^{-p} 
\|g\|_{L^p(\bbT^d)}^p.
\qedhere\]

\smallskip

\noi{\it Remark.} It is also possible to build a proof of Theorem \ref{LpthmTd} from Theorem
\ref{Lpthm}   using modifications of a duality argument by deLeeuw
\cite{deleeuw},
see also \cite{stein-weiss-book} and \cite{kenig-tomas}.

\section{Sharpness}\label{sharpnesssect}

In this section we show that Theorems   \ref{LpthmTd} and  \ref{Lpthm} fail for $q>p'$.
We shall first reduce the argument for Fourier series to the one for Fourier integrals by a familiar transplantation method and then modify an argument that was used by Tao to 
 obtain necessary conditions for the Bochner-Riesz maximal operator, see \cite[sect.5]{tao-wt},
 and  also the work  by Carbery and Soria \cite{carbery-soria} where a related argument  appears in the context of localization results for Fourier series.
Note that the almost everywhere convergence assertion in part (ii)  of Theorem \ref{LpthmTd} also fails for $q>p'$, by Stein-Nikishin theory (\cite{stein-annals61}).

\subsection{\it Fourier series}
 We have  for $f\in L^p(\bbT^d)$
\Be\label{supvsnorm}
\Big\| \sup_{T>0} \Big(\frac 1T\int_0^T |\cR^{\la }_t f|^q dt\Big)^{1/q} 
\Big\|_{L^{p,\infty}(\bbT^d)}
\ge 
\sup_{T>0}\Big\| \Big(\frac 1T\int_0^T |\cR^{\la }_t f|^q dt\Big)^{1/q} 
\Big\|_{L^{p,\infty}(\bbT^d)}\Ee
and our necessary condition will follow from Proposition \ref{necRdprop} below and the following result.
\begin{lemma}  Let $1\le p\le 2$.
Suppose that for some $C>0$
\Be\label{rieszTdassu}
\sup_{\|f\|_{L^p(\bbT^d)}=1}
\sup_{T>0}\Big\| \Big(\frac 1T\int_0^T |\cR^{\la }_t f|^q dt\Big)^{1/q} 
\Big\|_{L^{p,\infty}(\bbT^d)}\le C.
\Ee
Then also 
\Be\label{rieszRdconcl}\sup_{\|f\|_{L^p(\bbR^d)}=1}
\sup_{T>0}\Big\| \Big(\frac 1T\int_0^T |R^{\la }_t f|^q dt\Big)^{1/q} 
\Big\|_{L^{p,\infty}(\bbR^d)}\le C.
\Ee
\end{lemma}
\begin{proof}
By scaling, density of $C^\infty_c$ functions in $L^p$ and the monotone convergence theorem it  suffices to show for all $f\in C^\infty_c(\bbR^d)$, all compact sets $K$, all $\delta\in (0,1)$, all $\eps\in (0,1)$  and all $\alpha>0$
\[
\meas \Big(\Big\{x\in K: \Big(\int_\delta^1 |R^\la_t f(x)|^q dt\Big)^{1/q} >\alpha\Big\}\Big)
\le C^p(1-\eps)^{-p} \alpha^p \|f\|_p^p.\]
Fix such $f$, $\alpha$, $\delta$, $\eps$  and $K$. For large $L\in \bbN$ define
$$V^\la_{L,t} f(x) = \sum_{\ell\in \bbZ^d} L^{-d} \widehat f(L^{-1}\ell) 
(1-\rho(t^{-1}L^{-1} \ell )_+^\la e^{2\pi i L^{-1}\inn{x}{\ell}}.$$
Then $V^\la_{L,t}f(x)$ is a Riemann sum for the integral representing $R^\la_t f(x)$. Hence we have  $$\lim_{L\to\infty}V_{L,t}^\la f(x)= R^\la_t f(x)$$ with the limit uniform in $t\in [\delta,1]$, $x\in K$. We may therefore choose $L$ such that  $$\text{supp} (f(L\cdot))\,\subset \,\{x:|x|<1/4\} \,\text{ and } \, K\subset LQ^0$$  with $Q^0=[-1,2,1/2]^d$, and 
$$\sup_{\delta\le t \le 1}\sup_{x\in K} |R^\la_t f(x)- V^\la_{L,t}f(x)|<\alpha \eps.$$ It remains to show 
\Be\label{VLwt}
\meas \Big(\Big\{x\in K: \Big(\int_\delta^1 |V^\la_{L,t} f(x)|^q dt\Big)^{1/q} >\alpha(1-\eps)\Big\}\Big)
\le C^p(1-\eps)^{-p} \alpha^p \|f\|_p^p.\Ee
Observe that for $w\in Q^0$
$$\Big(\int_\delta^1 |V_{L,t }^\la f(Lw)|^q dt\Big)^{1/q}=
\Big(\frac 1L\int_{\delta L}^L \Big| \sum_{\ell\in \bbZ^d} L^{-d} \widehat f(L^{-1}\ell) (1-\rho(\ell/s))_+^\la e^{2\pi i  \inn{w}{\ell}} \Big|^q ds\Big)^{1/q}.
$$
Let $f_L^{\text{per}} (w)
=\sum_{\kappa\in \bbZ^d} f(L(w+\kappa))$. Then from    the Poisson summation formula the Fourier coefficients of the periodic function 
$f_L^{\text{per}} $ are given by  $\inn{f_L^{\text{per}}}{e_\ell}= L^{-d}\widehat f(L^{-1}\ell)$. Hence the expression on the right hand side of the last display is equal to
$(L^{-1}\int_{\delta L}^L| \cR^\la_t f^{\text{per}}_L |^qdt)^{1/q}$.
Replacing $K$ by the larger set $LQ_0$ and then changing variables $x=Lw$ we see that the expression on the left hand side of \eqref{VLwt} is dominated by 
\begin{align*}
&L^{d} \meas \Big(\Big\{w\in Q^0: \Big(\frac 1L \int_{\delta L}^L |
\cR^\la_{t} f^{\text{per}}_L(w)|^q dt\Big)^{1/q} >\alpha(1-\eps)\Big\}\Big)
\\ &\le L^{d} C^p \alpha^{-p}(1-\eps)^{-p} \int_{Q_0} |f^{\text{per}}_L(w)|^p dw,
\end{align*}
where the last inequality follows from assumption \eqref{rieszTdassu}.
Since the support of $f(L\cdot)$ is contained in $Q_0$ one immediately gets
$$\|f^{\text{per}}_L\|_{L^p(Q^0)}^p = \|f(L\cdot)\|_{L^p(\bbR^d)}^p=L^{-d} \|f\|_{L^p(\bbR^d)}^p.$$
This shows \eqref{VLwt} and concludes the proof.
\end{proof}

\subsection{\it Fourier integrals}
Using the $\bbR^d$ analogue of \eqref{supvsnorm} we reduce the sharpness claim in Theorem \ref{Lpthm}
 to the following proposition.
\begin{prop} \label{necRdprop} Let $1<p\le 2$, and $\la>-1/2$.
Assume that there is a constant $C>0$ such that 
\Be \label{nec}
\sup_{T>1}\Big\| \Big(\frac 1T\int_0^T |R^{\la }_t f|^q dt\Big)^{1/q} 
\Big\|_{L^{p,\infty}(\R^d)}
\leq C \| f \|_{L^{p}(\R^d)}
\Ee for all Schwartz functions $f$.
Then 
\[ \la \geq \la(p) + \frac{1}{2}\left(\frac{1}{p'}-\frac{1}{q}\right). \]
In particular, if \eqref{nec} holds for $\la = \la(p)$, then $q \leq p'$.
\end{prop}

\begin{proof}
We note that the inequality with a given $\rho$ is equivalent to the inequality with $\rho\circ A$ where $A$ is any rotation.

Let  $\xi^\circ\in \Sigma_\rho$ such that $|\xi^\circ|$ is maximal. Then 
the Gaussian curvature does not vanish at $\xi^0$. Choose small neighborhoods $U_1, U_0$ of $\xi^0$ in $\Sigma_\rho$ such that 
$\overline U_1\subset U_0$, 
the  Gauss map is  injective in a neighborhood of $\overline U_0$ and the curvature is bounded below on $U_0$. Let $\gamma$ be homogeneous of degree zero, $\gamma(\xi)\neq 0$ for $\xi\in U_1$ with $\gamma$ supported on the closure of the cone generated by $U_0$. 
Let $n(\xi_0)=\frac{\nabla\rho(\xi^\circ)}
{|\nabla\rho(\xi^\circ)|}$ the outer normal at $\xi_0$, let $\Gamma_\eps=\{x\in \bbR^d: \big|\frac{x}{|x|}-n(\xi_0)\big|
\le 2\eps\}$, with $\eps$ so small that 
$\Gamma_\eps$ is contained in the cone generated by the normal vectors  $\nabla\rho(\xi)$ with $\xi\in U_1$. Let, for $R\gg 1$,  $\Gamma_{\eps,R}=\{x\in\Gamma_\eps:|x|\ge R\}$.
By the choice of $\eps$ there is,  for each $x\in \Gamma_\eps$,  a unique $\Xi(x)\in \Sigma_\rho$, so that $\gamma(\Xi(x))\neq 0$ and so that $x$ is normal to $\Sigma_\rho$ at $\Xi(x)$. Clearly $x\mapsto\Xi(x)$ is homogeneous of degree zero on $\Gamma$, smooth away from the origin.
By a rotation we may assume
\Be\label{rotation} n(\xi^\circ)= (0,\dots,0,1).\Ee

By \S\ref{origincontr} inequality \eqref{nec} also implies the similar inequality where $R^\la_t f$ is replaced with $S^\la_t f$ and $S^\la_t$ is as in \eqref{Slatdef}.
 Let 
 $h_\la(s)= (1-\upsilon_0(s)) (1-s)_+^\la$
 and $$ K_{\la,t}  (x)= t^d \cF^{-1} [ \gamma \,h_\la\!\circ \!\rho](tx).$$ 
 Observe that 
 $K_{\la,t}*f=S^\la_t f_\gamma$ with $\widehat f_\ga=\gamma \widehat f$. By the H\"ormander multiplier theorem $\gamma$ is a Fourier multiplier of $L^p$ and we see 
 that \eqref{nec} implies that
 \Be \label{necK}\sup_{T>0}
\Big\| 
 \Big(\frac 1T\int_0^T |K_{\la,t} *f|^q dt\Big)^{1/q} 
\Big\|_{L^{p,\infty}(\R^d)}
\leq C \| f \|_{L^{p}(\R^d)}.
\Ee 

We now derive an asymptotic expansion for $K_{\la, 1}(x)$ when $x\in \Gamma_{\eps,R}$.
Recall that $\rho $ is homogeneous of degree $1/b$, i.e. $\rho(t^b\xi)=t\rho(\xi)$. 
We use generalized  polar coordinates $\xi=\rho^b\xi(\om)$ where $\om\to \xi(\om)$ is a parametrization of $\Sigma_\rho$ in a neighborhood of $U_0$. 
Then \begin{align*} d\xi&= b\rho^{db-1} d\rho\,\inn{\xi(\om)}{n(\xi(\om))} 
\big(\det (\frac{\partial\xi}{\partial\om})^\intercal
\frac{\partial\xi}{\partial\om}\big)^{1/2}\, d\om
\\
&= \rho^{db-1} d\rho\,\,|\nabla\rho(\xi')|^{-1} d\sigma(\xi'), \quad \xi'=\xi(\om).
\end{align*} 
Here we have used Euler's homogeneity relation $b\inn{\xi}{\nabla\rho(\xi)}=\rho(\xi)$ for vectors on $\Sigma_\rho$.
Then 
\Be\label{polar}
K_{\la,1}(x)
=\int_0^\infty h_\la(\rho)\rho^{bd-1}\int_{\Sigma_\rho}\gamma(\xi') e^{2\pi i\rho\inn{\xi'}{x}} 
\frac{d\sigma(\xi')}{|\nabla\rho(\xi')|} 
\,d\rho\,.
\Ee
We use the method of stationary phase and get for 
 $x\in \Gamma_{\eps, R}$
\Be \label {gahdec}K_{\la,1}(x) = I(x)+\sum_{j=1}^N II_j(x)+III(x),\Ee
where
\begin{align*}
 I(x)&=c\int_0^\infty h_\la(\rho)\rho^{{bd-1-\frac{d-1}{2}}} e^{2\pi i\rho\inn{\Xi(x)}{x}}  d\rho \,\,
 \frac{\gamma(\Xi(x)) |\nabla\rho(\Xi(x))|^{-1}  }
{( \inn{\Xi(x)}{x})^{\frac{d-1}{2}} |\text{curv}(\Xi(x))|^{1/2} },
\end{align*}
where $\text{curv}(\Xi(x))$ is the Gaussian curvature at $\Xi(x)$ 
and $c\neq 0$, and
\begin{align*}
 II_j(x)&=c_j\int_0^\infty h_\la(\rho)\rho^{{bd-1-\frac{d-1}{2}-j}} e^{2\pi i\rho\inn{\Xi(x)}{x}}  d\rho \,\,
 \frac{\gamma_j(\Xi(x))   }
{( \inn{\Xi(x)}{x})^{\frac{d-1}{2}+j} |\text{curv}(\Xi(x))|^{1/2} },
\end{align*}
where $\gamma_j$ is smooth.
For the remainder term we get 
$$| III (x) |\lc_N \|h\|_1|x|^{-N}\,, \quad x\in \Gamma_{\eps, R}\,.$$
In the resulting $\rho$ integrals we use asymptotics for the one-dimensional Fourier transform of 
$h_\la$,
 \cf. \cite[\S 2.8]{erdelyi},
and see that for $x\in \Gamma_{\eps, R}$,
$$
\int_0^\infty h_\la(\rho)
\rho^{{bd-\frac{d+1}{2}}} e^{2\pi i\rho\inn{\Xi(x)}{x}}  d\rho \,\, = C_\la \inn{\Xi(x)}{x}^{-\la-1} 
e^{2\pi i \inn{\Xi(x)}{x}}
+ O(\inn{\Xi(x)}{x})^{-\la-2}, 
$$
with similar asymptotics for the $\rho$-integrals in the  terms $II_j$.

Now set \, for $x\in \Gamma_\eps$,  $H(x)=\inn{\Xi(x)}{x}$ and use 
Euler's homogeneity relation  to write
$$H(x)
= |x|  \biginn{\Xi(x)}{ \frac{\nabla\rho(\Xi(x))}{|\nabla\rho(\Xi(x))|}}=
|x|\frac{\rho(\Xi(x))}{b|\nabla\rho(\Xi(x))|} = 
\frac{|x|}{b|\nabla\rho(\Xi(x))|}. $$  
If $\eps$ is small we then have for $ t |x|\gg R$,
$$ K_{\la,t}(x) = A_{\la} (x,t) + B_{\la}(x,t), \quad |x'|\le \eps^2|x_d|,$$
where
\begin{align}\notag
A_{\la}(x,t) \,=\,&
C(\la) 
 t^{d-\frac{d+1}{2}-\la}  G(x)
 e^{2\pi i t H(x)},
 \\
 \label{Gxdef}
 &\text{ where } G(x) = H(x)^{-\frac{d+1}{2}-\la} \, \frac{\gamma(\Xi(x)) |\nabla\rho(\Xi(x))|^{-1}  }
 {|\text{curv}(\Xi(x))|^{1/2} }
 \end{align}
 and \[ B_\la(x,t) \lc t^{d-\frac{d+3}{2}-\la}  H(x)^{-\frac{d+3}{2}-\la}  .\]

Recall 
\eqref{rotation} and split 
 $y=(y',y_d)$. 
We now let 
\[ P(T, \eps)= \{y: |y'|\le  T^{-1} \eps, \,\,|y_d|\le T^{-1/2}\eps\} \]
 and define 
$$f_T(y) =  \bbone_{P(T,\eps)}(y)
e^{2\pi i \eps T y_d}.$$
Then \Be\label{fTLp}\|f_T\|_p\lc T^{\frac 1p(\frac 12-d)}.\Ee
We examine the integrals $K_{\la, t} * f_T(x) $ for $|x|\approx 1$ and $R\ll t \approx \eps T$.
We may obtain a lower bound for the absolute value of  this  integral if we can choose $t$ for given $x$ such that
\Be  \label{tchoice} 2\pi (\eps Ty_d+tH(x-y)-tH(x) )\in (-\frac{\pi}{4},\frac \pi 4) \text{ for all $y\in \text{supp}( f_T)$.} \Ee
As the Gauss map is invertible near $\xi^\circ$ we observe that $H$ is smooth and homogeneous of degree $1$.  We have $\nabla H(x)= \xi^\circ+O(\eps)$ and thus $\partial_{x_d}H(x)\ge c>0$.
Now  
\begin{multline}\label{phaseappr}
\eps Ty_d+tH(x-y)-tH(x) 
= \\-t\sum_{i=1}^{d-1} y_i \partial_{x_i}H(x)
+y_d (\eps T- t \partial_{x_d}H (x)) +t\sum_{i,j=1}^d y_iy_j \int_0^1 (1-s) \partial^2_{x_ix_j}H(x-sy) ds.
\end{multline}
The first and the third term on the right hand side are $O(\eps)$ when $y\in \text{supp}( f_T)$.
We choose $t$ in the interval
\Be\label{IxT} I_{x,T}= \Big[\frac{\eps T}{\partial_{x_d}\!H(x)}-\eps T^{1/2} ,
\frac{\eps T}{\partial_{x_d}\!H(x)}+\eps T^{1/2} \Big].
\Ee
We assume that $\eps$ is chosen so small  that $I_{x,T}\subset [0,T]$.
If $t\in I_{x,T}$ the second term on the right hand side of \eqref{phaseappr} will be $O(\eps)$ as well so that \eqref{tchoice} is satisfied.

We now split
\[ K_{\la,t}*f_T(x)= \sJ_1(x,t)+ \sJ_2(x,t) +\sJ_3(x,t)\]
with
\begin{align*}
\sJ_1(x,t)&= C(\la) G(x) e^{2\pi itH(x)} t^{\frac{d-1}{2}-\la} 
\int e^{2\pi i (T\eps y_d+tH(x-y)-tH(x))} \bbone_{P(T,\eps)}(y) \, dy,
\\
 \sJ_2(x,t)&= C(\la)   t^{\frac{d-1}{2}-\la} 
\int e^{2\pi i (T\eps y_d+tH(x-y))} (G(x-y)-G(x)) \bbone_{P(T,\eps)}(y) \, dy,
\\ \sJ_3(x,t)&= \int B_\la(x-y,t) f_T(y) \,dy\,.
\end{align*}

We estimate these terms for \Be\label{xt-assu}x\in \Omega:=\{x: |x'|\le \eps^2|x_d|,\,\, 1/2\le|x_d| \le 1\}, \quad 
t\in I_{x,T}.\Ee Then by \eqref{tchoice} 
the real part of the integrand in the definition of $\sJ_1(x,t)$ is bounded below by 
$2^{-1/2}\bbone_{P(T,\eps)}(y)$ and therefore, for $x\in \Omega$,
\begin{align*} 
|\sJ_1(x,t)| &\ge C G(x) t^{\frac{d-1}{2}-\la}
\int \bbone_{P(T,\eps)}(y)
\, dy
\\&\ge c t^{\frac{d-1}{2}-\la} T^{\frac 12-d}.
\end{align*} Moreover, \begin{align*}
|\sJ_2(x,t)|&\lc  t^{\frac{d-1}{2}-\la}   \eps T^{-1/2} T^{\frac 12-d}
\\
|\sJ_3(x,t)|&\lc  t^{\frac{d-3}{2}-\la}   T^{\frac 12-d}.
\end{align*}
Hence  for small $\eps$ and $t|x|\gg R$, $t\in I_{x,T}$ the term $|\sJ_1(x,t)| $ is significantly larger than the terms $|\sJ_2(x,t)|$ and $|\sJ_3(x,t)|$.
Consequently, by $| I_{x,T}|\ge \eps T^{1/2}$, and assuming 
\eqref{xt-assu} we get 
\begin{align*}
&\Big(\frac 1T\int_0^T |K_{\la,t}* f_T(x)|^q dt\Big)^{1/q} \ge
\Big(\frac 1T\int_{I_{x,T} }|K_{\la,t}* f_T(x)|^q dt\Big)^{1/q}
\\ &\ge c \eps^{1/q} T^{-1/2q} (\eps T)^{\frac{d-1}{2}-\la} T^{1/2-d} = c_\eps T^{-\frac d2-\la-\frac{1}{2q}}
\end{align*} 
and thus
\[\Big\|\Big(\frac 1T\int_0^T |K_{\la,t}*f_T|^qdt\Big)^{1/q} \Big\|_{L^{p,\infty}}
\gc_\eps 
T^{-\frac d2-\la-\frac{1}{2q}} T^{\frac dp-\frac1{2p}}  \|f_T\|_p
\]
which for  $T\to \infty$  implies  $\la\ge \la(p) + \frac12(1-\frac 1p-\frac 1q)$.
\end{proof}



\section{An $L^1$ result}\label{L1resultsect}
We currently do not have an analogue of Theorem \ref{H1thmTd} for general functions in $L^1(\bbT^d)$. We  formulate a weaker result which is essentially a consequence of Theorem \ref{H1thmTd}.

\begin{thm}\label{L1thm} (i) Let $f\in L^1( \bbT^d)$. Then for all $q<\infty$, and $\la(1)=\frac{d-1}{2}$,
$$\lim_{T\to \infty} \Big\| \Big(\frac 1T\int_0^T |\cR^{\la(1)}_t f-f|^q dt\Big)^{1/q}\Big\|_{L^{1,\infty}(\bbT^d)}=0.$$

(ii) The analogous statement holds on $L^1(\bbR^d) $ with $R^{\la(1)}_tf$ in place of $\cR^{\la(1)}_t f$.
\end{thm}

\begin{proof}  Since the convergence holds for Schwartz function one can by a standard approximation argument  reduce the proof of  (ii)  to the inequality
\Be\label{L1ineq}\sup_{T>0} \Big\| \Big(\frac 1T\int_0^T |R^{\la(1)}_t \!f|^q dt\Big)^{1/q}\Big\|_{L^{1,\infty}(\bbR^d)}\lc \|f\|_{L^1(\bbR^d)}. \Ee
Similarly the proof of (i) is reduced to a  corresponding inequality on $\bbT^d$, with the supremum in $T$ extended over $T\ge 1$. The weak type $(1,1)$ inequality in the $\bbT^d$ case  follows from the $\bbR^d$  case  by the transference arguments of \S\ref{transferencesect}. Therefore, it suffices to show \eqref{L1ineq}.

By  the maximal estimate in \S\ref{origincontr}  it remains to prove
\Be\label{L1ineqtrunc} \Big\| \Big(\frac 1T \int_0^T |S^{\la(1)}_t \!f|^q dt\Big)^{1/q}\Big\|_{L^{1,\infty}(\bbR^d)}\lc \|f\|_{L^1(\bbR^d)},\Ee
where $S^{\la(1)}_t$ is as in \eqref{Slatdef}.
We may assume $q\ge 2$. Now 
\begin{align*}
\Big(\frac 1T \int_0^T |S^{\la(1)}_t \!f (x)|^q dt\Big)^{1/q}
\le \sum_{l=0} 2^{-l/q} \Big( \frac{1}{T2^{-l}}\int_{T2^{-l}}^{T2^{-l+1}} |S_t^{\la(1)}\!f(x)|^q dt\Big)^{1/q}
\end{align*}
and we claim the inequality 
\Be\label{Ascaled} 
\sup_{A>0}\Big\|\Big( \frac{1}{A}\int_{A}^{2A} |S_t^{\la(1)}\!f|^q dt\Big)^{1/q}\Big\|_{L^{1,\infty}} \le C_q \|f\|_1.
\Ee
Assuming  that \eqref{Ascaled}  is verified  we can deduce that the left hand side of \eqref{L1ineqtrunc}
is bounded by $C_q \tilde C \sum_{l>0} (1+l) 2^{-l/q} \|f\|_1 \lc_q\|f\|_1$, by the theorem  of Stein and N. Weiss  \cite[Lemma 2.3]{stein-nweiss} on summing $L^{1,\infty}$ functions.  

Let $\eta$ be as \eqref{eta=1}.
Then our main result, Theorem \ref{mainthm}, yields for all $A>0$
\[
\Big\|\Big( \frac{1}{A}\int_{A}^{2A} |S_t^{\la(1)}\!f|^q dt\Big)^{1/q}\Big\|_{L^{1,\infty}(\bbR^d)} \le C_q 
\| \cF^{-1} [\eta(A^{-1} \cdot)] * f \|_{H^1(\bbR^d)}.
\] 
Since $\eta $ is $C^\infty$ and compactly supported away from the origin we have 
\[ \| \cF^{-1} [\eta(A^{-1} \cdot)] * f \|_{H^1(\bbR^d)} \lc \|f\|_{L^1(\bbR^d)}\]
uniformly in $A$. This yields \eqref{Ascaled}  and concludes the proof of \eqref{L1ineqtrunc}.
\end{proof}
As an immediate consequence of  Theorem \ref{L1thm} we get 
\begin{cor} \label{strongqcor}
Let $f\in L^1(\bbT^d)$. There is a subsequence  $T_j\to \infty$ such that
\Be\label{subsequence} \lim_{j\to\infty} \Big(\frac{1}{T_j} \int_0^{T_j} |\cR^{\la(1)}_t \!f(x)-f(x) 
|^q dt\Big)^{1/q} =0 \text{ a.e.}\Ee
\end{cor}
Arguing as 
in \cite[ch. XIII.7] {zygmund}  or \cite[\S4]{bwilson}
we get
\begin{cor} \label{almostalmost}
 Let $f\in L^1(\bbT^d)$. For almost every $x\in \bbT^d$ there is a measurable set $E=E(f,x)$ of upper density one, i.e. satisfying
 \Be\label{limsupconcl}
 \limsup_{T\to \infty} \frac{|E\cap [0,T]|}{T}=1\Ee
 such that  
 $$\lim_{\substack{ t\to \infty\\ t\in E} } \cR_t^{\la(1)}\!f(x)=f(x).$$
 \end{cor}
 For convenience of the reader we give  a proof.
 \begin{proof} 
  Fix $x$ such that  \eqref{subsequence} in Corollary \ref{strongqcor} holds and let $g(t)= |\cR^{\la(1)}_t \!f(x)-f(x) |$.
  We may assume that $T_j$ is increasing in $j$.
    For $m=1,2, \dots$ let $E_m=\{t: g(t)\le 1/m\}$.  By Tshebyshev's inequality we have
    \[
    \frac {|E_m^\complement\cap [0, T_j]|}{T_j} 
    \le m^q \frac 1{T_j}\int_0^{T_j} g(t)^q dt  
    \]
    which by assumption tends to $0$ as $j\to \infty$. Hence
    $\lim_{j\to \infty} T_j^{-1} |E_m\cap [0, T_j] |=1$. 
    Thus we may choose a strictly  increasing sequence $j_m$ of positive integers such that 
 $T_j^{-1} |E_m\cap [0, T_j] |> 1-m^{-1}$ for $j\ge j_m$.
 Let $E=[0, T_{j_1} ]\cup  \bigcup_{m=1}^\infty  (E_m\cap [T_{j_m}, T_{j_{m+1}}])$. Since the sets $E_m$ are decreasing we have
 $$|E\cap [0, T_{j_{m+1}}]| \ge |E_m\cap [0, T_{j_{m+1}}]|  \ge (1-m^{-1} ) T_{j_{m+1}}$$ and hence 
 $\limsup_{T\to \infty} T^{-1} |E\cap [0,T]|=1$. 
 Now $E\cap [T_{j_m},\infty]\subset E_m$ and thus  $g(t)\le m^{-1}$ on this set. It follows that  $g(t)\to 0$ as $t\to \infty$ within $E$.
 \end{proof}

It would be desirable to replace the $\limsup$   in \eqref{limsupconcl} by the $\liminf$.  
The proof of the corollary shows that this would require the existence a.e. of the limit in 
\eqref{subsequence} for {\it all } sequences $T_j\to \infty$. We can currently prove this only for functions in $h^1$.

\section{Maximal functions on $H^p(\bbR^d)$ for $p<1$}\label{Hpsect}
We now consider the maximal  operator associated with the  generalized Riesz means 
when they act on  functions or distributions in the Hardy space $H^p(\bbR^d)$ for $p<1$. The following result   generalizes one  by Stein, Taibleson and Weiss \cite{stein-taibleson-weiss} for  the standard Bochner-Riesz means. Other  generalizations for specific rough $\rho$ were considered in  \cite{hong-kim-yang} and the references therein.

Let $R^\la_t$ be as in \eqref{rieszmeansRd}.

\begin{thm} \label{Hpthm}
For $0<p<1$, $\la(p)=d(1/p-1/2)-1/2$ we have for all $f\in H^p(\bbR^d)$
\[\big\|\sup_{t>0} |R^{\la(p)}_t\! f| \big\|_{L^{p,\infty}(\bbR^d)} \lc \|f\|_{H^p(\bbR^d)}.\]
\end{thm}

We use the same reductions as in \S\ref{mainwtsect}. Write, for $t>0$
$$
R^\la_t f(x) = \cF^{-1}[u(\rho(\cdot/t)) \widehat f](x) + \sum_{j= 1}^\infty 2^{-j\la} T_j f(x,t),
$$
where $u$ is as in \S\ref{origincontr} and 
$\widehat {T_j f}(\xi,t)= \vphi_j(\rho(\xi/t))\widehat f(\xi)$ 
with $\phi_j$ as in \S\ref{Littlewood-Paley}.
This is similar to \eqref{Tjkdef} (except that now $t$ ranges over $(0,\infty)$).
The functions $u$, $\varphi_j$ depend on $\la$ but satisfy uniform estimates as $\la$ is taken over a compact subset of $\bbR$.
Let
\[\cM_0 f(x)= 
\sup_{t>0}|\cF^{-1}[u(\rho(\cdot/t)) \widehat f](x)\]
and for $j\ge 1$, $$\cM_j f(x)= \sup_{t>0} |T_jf(x,t)|.$$
We then have 
\Be \label{ptwmax}\sup_{t>0} |\cR^{\la(p)}_t f(x)|\le \cM_0 f(x)+ \sum_{j\ge 1} 2^{-j\la(p)}  \cM_j f(x),
\Ee and we shall derive a weak type inequality on $H^p$ for the right hand side in \eqref{ptwmax}.
The ingredients are $H^p\to L^p$ bounds for the maximal operators 
$\cM_0$ and $\cM_{j}$. 

Let $M$ be a nonnegative integer.
We recall that a function $a$ supported on a ball $B$ is a $(p,M)$ atom if $\|a\|_\infty \le \text{vol}(B) ^{-1/p}$ and $\int a(x)P(x) dx=0$ for all polynomials of degree at most $M$. By the atomic decomposition it suffices to check the $H^p \to L^p$ bounds on $(p,M)$ atoms for every non-negative integer $M > d(p^{-1}-1)-1$. The bound for $\cM_0a$ is straightforward:

\begin{lemma} \label{originHardythm} Let $M+1> d(p^{-1}-1)$ and let $a$ be a $(p,M)$-atom.
 For $0<p\le 1$ we have 
$$
\|\cM_0a 
\|_p\lc 1.
$$
\end{lemma} 

\begin{proof} This follows by a variant of  the argument in \S\ref{origincontr}.
Define  $$\fM_{N_1,N_2}f=
\sup_{\tau>0} |\cF^{-1}[ (1-\rho(\cdot/\tau)^{N_1})^{N_2}_+\widehat f]|.$$ 
Let $N_1$, $N_2$ be large so that $\fM_{N_1,N_2}$ maps $H^p$ to $L^p$.
By the subordination formula \eqref{subord} we have
\Be
\sup_{t>0}|\cF^{-1}[u(\rho(\cdot/t))\widehat a](x)|
\lc |\fM_{N_1,N_2} a (x)| 
\frac{1}{N_2!}\int_0^\infty
s^{N_2} |u_{N_1}^{(N_2+1)} (s) |ds\Ee
and the integral is finite.
Hence we get the desired $L^p$ bound for $\cM_0a$.
\end{proof}

\subsection{\it The main $H^p\to L^p$ bound}

\begin{prop} \label{Mjestthm} Let  $0<p\le 1$,  $j\ge 1$, $\nu\in \cZ_j$.
Let $M+1> d(p^{-1}-1)$ and let $a$ be a $(p,M)$-atom. Then
\[ \big\|\cM_{j} a\big\|_p\lc 2^{j(d(\frac 1p-\frac 12)-\frac 12)}.
\]
\end{prop}

We further decompose 
$T_jf(x,t)=\sum_{\nu\in \cZ_j} T_{j,\nu} f(x,t)$ where we use the homogeneous partition of unity as in \eqref{Tjknudef}.
Let  for $\nu \in \cZ_j$,
\begin{align*}
\cM_{j,\nu}  f(x)&= \sup_{t>0} |T_{j,\nu} f(x,t)|.
\end{align*} Then
$\cM_j f(x)
\le \sum_{\nu\in \cZ_j} \cM_{j,\nu} f(x)$.
Since $\#\cZ_j=O(2^{j(d-1)/2})$ 
we can use the 
triangle inequality in $L^p$, $p\le 1$, to see that 
the proposition follows from
\Be\label{Mjnuest}
\big\|\cM_{j,\nu} a\big\|_p\lc 2^{j\frac{d+1}{2}(\frac 1p-1)} .
\Ee
We proceed with the proof of
\eqref{Mjnuest}.

By translation and scaling, we may assume that $a$ is supported in the ball $B$ of radius $1$ centered at the origin, $ \|a\|_\infty\le 1$ and $\int a(x)P(x) dx=0$ for all polynomials of degree $\le M$.
By a rotation we may also assume that $\nabla\rho (\xi_{j,\nu})$ is parallel to $(1,0,\dots,0)$ and 
thus writing  $x=(x_1,x')$ we have 
\begin{equation}\label{eqn:ptbound}
 |\partial^\alpha K_{j,\nu} (x)| \leq C_{N_1,N_2,\alpha} \frac{
 2^{-j(d+1)/2 }}{(1+2^{-j} |x_1|)^{N_1}(1+2^{-j/2} |x'|)^{N_2}}\,
\end{equation}
for all  multiindices $\alpha\in \bbN_0^d$ and  all $N_1, N_2\ge 0$; \cf.  \cite{christ-sogge-survey} or \cite{seeger-archiv}.
Let 
\[ \sD= \{ (x_1,x')\in \R^d : |x_1| \leq  5\cdot 2^j, |x'| \leq 5 \cdot 2^{j/2}\}. \]
In the following subsections, we estimates the $L^p$-quasi-norm of  $\cM_{j,\nu} a(x)$  over $\sD$ and $\sD^\complement$, respectively, using the cancellation condition for the atom when $x\in \sD^\complement$.

\subsubsection{Estimation over $\sD$}
Let
\begin{align*}
\sD_0& = \{ (x_1,x')\in \R^d : |x_1| \leq  5, |x'| \leq 5\} 
\\
\sD_1 &= \{ (x_1,x')\in \R^d : |x_1| \leq  5\cdot 2^{j/2}, |x'| \leq 5\}
\end{align*} and 
\Be\label{Edef} E= \{
 (x_1,x') \in \R^d : |x'|\geq 2^{-j/2} |x_1| \}.
\Ee

We derive  the following pointwise estimates
\Be\label{ptwmaxfct}
\cM_{j,\nu} a(x) \lc \begin{cases}
1  &\text{ if } x\in\sD_0,
\\
 2^{\epsilon j/2} |x_1|^{-(1+\epsilon)}  &\text{ if } x\in \sD_1\setminus \sD_0,
 \\
 2^{-j/2} |x'|^{-d}  &\text{ if } x \in (\sD\setminus \sD_1)\cap E,
 \\
 2^{j\frac{d-1}2} |x_1|^{-d} &\text{ if } x \in (\sD\setminus \sD_1)\cap E^\complement.
\end{cases} 
\Ee
If we use this for  $0<\eps<\frac 1p-1$ then straightforward integrations give the desired bound
\Be\label{LpDbd} \|\cM_{j,\nu} a\|_{L^p(\sD)}\lc  2^{j\frac{d+1}{2} (\frac 1p-1)}.\Ee

To verify \eqref{ptwmaxfct} first observe the pointwise bound $\cM_{j,\nu} a(x) 
\le \sup_{t>0} \|t^d K_{j,\nu} (t\cdot)\|_1 \|a\|_\infty \lc 1.$
This gives  \eqref{ptwmaxfct} for $x\in \sD_0$.
Secondly 
for any $x\in \sD_1 \setminus \sD_0$ and $y\in B_1(0)$, we have 
$|x_1- y_1| \gc|x_1|$.
Using \eqref{eqn:ptbound} with $N_1=1+\epsilon$ and $N_2=d-1-\epsilon$, we have
\begin{align*}
 |t^d K_{j,\nu} (t\cdot)*a (x) | &\les t^d 2^{-j(d+1)/2}  (2^{-j} t |x_1|)^{-(1+\epsilon)} \int_{|y'|\leq 1} (2^{-j/2}t|x'-y'|)^{-(d-1-\epsilon)} dy' \\
 & \les 2^{\epsilon j/2} |x_1|^{-(1+\epsilon)} 
\end{align*} 
for all $x\in \sD_1\setminus \sD_0$.



Assume that $x\in (\sD\setminus \sD_1)\cap  E$. Then $|x'|  \geq 5$ and thus $|x'-y'| \ge c|x'|$ for some $c>0$ for all $|y'|\leq 1$. Using \eqref{eqn:ptbound} with $N_1=0$ and $N_2=d$, we have
\begin{align*}
 |t^d K_{j,\nu} (t\cdot)*a (x) | \les t^d 2^{-j(d+1)/2} (2^{-j/2}t|x'|)^{-d} = 2^{-j/2} |x'|^{-d}.
\end{align*} 
Finally, when $x\in (\sD\setminus \sD_1)\cap E^\complement$, we have $|x_1-y_1| \geq c |x_1|$ and necessarily $|x_1| \ge 2^{-j/2}$. If we put $N_1=d$, $N_2=0$ in \eqref{eqn:ptbound} we get 
\begin{align*}
 |t^d K_{j,\nu} (t\cdot)*a (x) | \les t^d 2^{-j(d+1)/2} (2^{-j}t|x_1|)^{-d} = 2^{j(d-1)/2} |x_1|^{-d}.
\end{align*} 
This concludes the proof of the pointwise estimate \eqref{ptwmaxfct} which implies \eqref{LpDbd}.



\subsubsection{Estimation over $\sD^\complement$}
When $x\in \sD^\complement$ we use the cancellation of the atom and Taylor's formula to write
\begin{align*}
 t^d K_{j,\nu} (t\cdot) * a(x) &=  t^d\int \Big(K_{j,\nu} (tx-ty)-\sum_{n=1}^M \frac{\inn{-ty}{\nabla}^nK_{j,\nu}( tx)}{ n!} \Big)a(y)\, dy
 \\
 &=\frac{(-1)^{M+1} }{M!}  t^{d+M+1}\int_0^1(1-s)^M  \int \inn{y}{\nabla}^{M+1} K_{j,\nu} (tx-sty) a(y) dy\, ds.
 \end{align*}
We now use \eqref{eqn:ptbound}
 for the derivatives of order $M+1$.
Also notice that if $E$ is as in \eqref{Edef} we have 
$|x'|\ge 5\cdot 2^{j/2}$ for $x\in \sD^\complement\cap E$ 
and $|x_1|\ge 5\cdot 2^j $ for $x\in \sD^\complement\cap E^\complement$.
We obtain
\Be\notag
\cM_{j,\nu} a(x) \lc \begin{cases}
2^{jM/2} |x'|^{-d-1-M}  &\text{ if } x\in\sD^\complement\cap E,
\\
2^{j (M+\frac{d+1}{2})} |x_1|^{-d-M-1} &\text{ if } x\in \sD^\complement\cap E^\complement,
 \end{cases} 
 \Ee
 where for $x\in\sD^\complement\cap E$ we took $N_1=0$, $N_2=d+M+1$ in \eqref{eqn:ptbound}
 and for 
$x\in \sD^\complement\cap E^\complement$ we took $N_1=d+M+1$ and $N_2=0$.
Hence
\[
\|\cM_{j,\nu} a\|_{L^p(\sD^\complement\cap E)} \lc 2^{jM/2}\Big(\int\limits_{|x'|\gc2^{j/2}}|x'|^{-(d+1+M)p}\int\limits_{|x_1|\lc 2^{j/2}|x'|} dx_1 dx'\Big)^{1/p}
\]
and 
\[
\|\cM_{j,\nu} a\|_{L^p(\sD^\complement \cap E^\complement)} \lc 2^{j(M+\frac{d+1}{2})}
\Big(\int\limits_{|x_1|\gc2^{j}}|x_1|^{-(d+1+M)p}\int\limits_{|x'|\lc 2^{-j/2}|x_1|} dx' dx_1\Big)^{1/p}.
\]
Both integrals are   $\lc 2^{j\frac{d+1}{2}(\frac 1p-1)}$ provided $p>\frac{d}{d+1+M}$,
which is the hypothesis on $p$ and $M$. This concludes the proof of 
\eqref{Mjnuest}.
\qed


\subsection{\it Proof of Theorem \ref{Hpthm}, conclusion}
As a crucial ingredient we shall use the generalized triangle inequality for $L^{p,\infty}$,
namely
\Be\label{triangle}
\Big\|\sum_l f_l\Big\|_{L^{p,\infty}} \lc A_p\Big(\sum_l \|f_l\|_{L^{p,\infty}}^p\Big)^{1/p},
\Ee
which holds with $A_p=O( (1-p)^{-1/p}) $ for $0<p<1$.
See either the paper by  Kalton \cite{kalton} or the paper by   Stein-Taibleson-Weiss \cite{stein-taibleson-weiss}.  By Lemma \ref {originHardythm} it suffices to prove
\Be\label{jge1}
\Big\|  \sum_{j\ge 1} 2^{-j\la(p)} \cM_{j} f
\Big\|_{L^{p,\infty}}\lc \|f\|_{H^p},
\Ee 
and by 
\eqref{triangle} and the atomic decomposition we may assume that $f$ is a 
$(p,M)$-atom $a$,  with $M+1> d(p^{-1}-1)$.
By dilation and translation invariance we may assume that $a$ is  function supported  in 
$\{x:|x|\le 1\}$ such that $\|a\|_\infty\le 1$ and  $\int a(x)P(x) dx=0$ for all polynomials of degree $\le M$. Because of this normalization we notice that (up to a harmless constant)  the function $a$ is also a $(p_1,M)$ and a $(p_0,M)$ atom where 
  $p_1<p<p_0<1$ and we pick $p_1$  sufficiently close to $p$ such that $M+1>d(p_1^{-1}-1)$.  

We need to verify for all  $\alpha>0$
\Be\label{desired}
\meas\big(\big\{x: 
\sum_{j\ge 1} 2^{-j\la(p)}  \cM_{j} a 
>\alpha\big\}\big) \lc \alpha^{-p}.
\Ee
By Proposition \ref{Mjestthm} we have
for every $j\ge 1$
\Be\label{Lpi}
\Big\|
2^{-j\la(p)}  \cM_{j} a \Big\|_{p_i}
\lc 2^{j (\la(p_i)-\la(p)) } .
\Ee
We employ a variant of an interpolation argument in \cite{bourgain-restrwt} to estimate
\[\meas\big(\big\{x: 
\sum_{j\ge 1} 2^{-j\la(p)}  \cM_{j} a 
>\alpha\big\}\big) 
\le\, I+II  ,
\]
where
$I$ is the measure of the set on which 
$\sum_{\substack{2^j\le  \alpha^{-p/d}}} 2^{-j\la(p)} \cM_{j} a 
>\alpha/2$ and $II$ is 
the measure of the set on which 
$\sum_{\substack{2^j> \alpha^{-p/d}}} 2^{-j\la(p)} \cM_{j} a 
>\alpha/2$. By Tshebyshev's inequality 
\begin{align*}
I&\le (2/\alpha)^{p_1} \Big\|\sum_{\substack{j\in \bbN:\\ 2^j\le  \alpha^{-p/d}}} 
2^{-j\la(p)}  \cM_{j} a \Big\|_{p_1}^{p_1} ,\\
II&\le (2/\alpha)^{p_0} \Big\|\sum_{\substack{j\in \bbN:\\ 2^j> \alpha^{-p/d}}} 2^{-j\la(p)}  \cM_{j} a \Big\|_{p_0}^{p_0} .
\end{align*}
Apply \eqref{Lpi} to obtain 
\begin{align*}
I+II&\lc\alpha^{-p_1}
\sum_{\substack{2^j\le  \alpha^{-p/d}}} 2^{j(\la(p_1)-\la(p))p_1}
+
\alpha^{-p_0}
\sum_{\substack{ 2^j>   \alpha^{-p/d}}} 2^{j(\la(p_0)-\la(p))p_0}
\\&=
\alpha^{-p_1}
\sum_{\substack{2^{jd} \le  \alpha^{-p}}} 2^{jd(1-\frac {p_1}{p})}
+
\alpha^{-p_0}
\sum_{\substack{ 2^{jd} >   \alpha^{-p}}} 2^{jd (1-\frac{p_0}{p})} \,\lc\,\alpha^{-p}.
\end{align*}
This yields \eqref{desired} and concludes the proof. \qed

\noi{\it Remark. }  Versions of the Fan-Wu transference argument  in \S\ref{h1transference} for maximal functions and $h^p$ for $p<1$ can be used to prove a theorem for Riesz means of Fourier series  analogous to Theorem \ref{Hpthm}, i.e. the maximal function 
$\sup_{t>0} |\cR^{d(1/p-1/2)-1/2}_t f|$ defines an operator that maps  $h^p(\bbT^d)$ to  $L^{p,\infty}(\bbT^d)$ when $p<1$.

\section{Open problems} \label{sec:open}

\subsection{\it Spaces near $L^1$}
For  $f\in L^1(\bbT^d)$ it remains open whether the Riesz means $\cR^{\la(p)}_t\!f(x)$
 converge $q$-strongly a.e. for any $q<\infty$. In particular can one upgrade 
in Corollary \ref{almostalmost} the conclusion of upper density one of  $E(f,x)$ to density one?

It may  also be interesting to investigate strong convergence a.e. for spaces  intermediate between $L^1$ and $L\log L$.

\subsection{\it The case $q=p'$}
For $f\in L^p(\bbT^d)$, $1<p<2$, prove or disprove that    $\cR^{\la(p)}_t\!f(x)$  converges $q$-strongly a.e.  when $q=p'$.
For $f\in h^1(\bbT^d)$, is there a version  of Rodin's theorem \cite{rodin}   in one dimension, that applies to  Riesz means at the critical index 
$\la(1)=\frac{d-1}{2}$ where the $L^q$-average  norm in  $t$-variable  is replaced by a $BMO$-average?

\subsection{\it Problems involving nonisotropic dilations} One can ask the same questions for quasi-radial Riesz means 
when the isotropic dilation group is replaced by a nonisotropic dilation group $t^P$ where $P$ is a matrix with positive eigenvalues
and $\rho$ satisfies $\rho(t^P \xi)=t\rho(\xi)$.
It turns out that the results depend on   the geometry of the surface in relation to the  eigenvectors of $P$.
In the case 
that  $\Sigma_\rho=\{\xi:\rho(\xi)=1\}$ 
has  nonvanishing curvature everywhere one has almost everywhere convergence for $\la>\frac{d-1}{2}$, but there are other examples where a.e. convergence fails for $\la<d/2$,
see \cite{yk-as} for details. Even in the  case of nonvanishing curvature we have currently 
no endpoint results for strong convergence of $\cR^\la_tf $, for the critical $\la=\la(p)$  when the  dilations are nonisotropic.

\subsection{\it Almost everywhere convergence} For $1<p<2$ the problem of a.e. convergence, and the critical $q$ for strong summability for $\la>\la(p)$ is wide open. 
Optimal results for the maximal operators are currently known only for the subspace $L^p_\rad$ of radial $L^p$ functions, see  \cite{gs-max}.
For general $L^p$ functions results that improve on Stein's classical theorem  for a.e. convergence of Riesz means of index $>(d-1)(1/p-1/2)$ are currently only known   in two dimensions, see   Tao's paper \cite{tao-max}.

\bibliographystyle{amsalpha}


\end{document}